\documentclass[11pt,leqno]{amsart}

\usepackage{amsmath,amsthm,amscd,amssymb,amsfonts, amsbsy}
\usepackage{latexsym}
\usepackage{txfonts}
\usepackage{exscale}
\usepackage{mathrsfs}
\usepackage{esint} % added to replace the $\nint$ with $\fint$
\usepackage{enumitem} 
\usepackage{color}
 \newcommand{\Bk}{\color{black}}

\definecolor{br}{rgb}{1, 0.4,0}

\usepackage[
  hmarginratio={1:1},     % equal left and right margins
  vmarginratio={1:1},     % equal top and bottom margins
  textwidth=400pt,			  % new text width
	textheight=580pt,        % new text height
  heightrounded,          % always useful
]{geometry}

\usepackage[utf8]{inputenc}

\usepackage[colorlinks,citecolor=blue,pagebackref,hypertexnames=false,pagebackref]{hyperref}

\usepackage{appendix}

\allowdisplaybreaks

\numberwithin{equation}{section}

\theoremstyle{plain}
\newtheorem{theorem}[equation]{Theorem}
\newtheorem{prop}[equation]{Proposition}
\newtheorem{corollary}[equation]{Corollary}
\newtheorem{lemma}[equation]{Lemma}

\theoremstyle{definition}
\newtheorem{defn}[equation]{Definition}

\theoremstyle{remark}
\newtheorem{remarks}[equation]{Remarks}
\newtheorem{remark}[equation]{Remark}

\numberwithin{equation}{section}

\newcommand{\RR}{{\mathbb{R}}}

\newcommand{\ZZ}{{\mathbb{Z}}}

\newcommand{\WW}{\mathcal{W}}

\newcommand{\C}{\mathcal{C}}
\newcommand{\DD}{\mathbb{D}}
\newcommand{\dd}{\mathbb{D}}
\newcommand{\PP}{\mathbb{P}}
\newcommand{\HH}{\mathcal{H}}
\newcommand{\F}{\mathcal{F}}
\newcommand{\N}{\mathcal{N}}

\newcommand{\W}{\mathcal{W}}

\newcommand{\dist}{\operatorname{dist}}

\newcommand{\pom}{\partial\Omega}

\newcommand{\hm}{\omega}

\renewcommand{\emptyset}{\mbox{\textup{\O}}}

\DeclareMathOperator*{\osc}{osc}

\DeclareMathOperator{\diam}{diam}
\DeclareMathOperator{\interior}{int}

\DeclareMathOperator*{\Lip}{Lip}

\def\XXint#1#2#3{{\setbox0=\hbox{$#1{#2#3}{\int}$}
     \vcenter{\hbox{$#2#3$}}\kern-.5\wd0}}

\DeclareMathOperator{\divg}{div}
\DeclareMathOperator{\Id}{Id}

\newcommand{\pO}{\partial\Omega}

\newcommand{\wcalA}{\mathcal{A}}

\reversemarginpar

\begin{document}
\allowdisplaybreaks

\title[Uniform rectifiability and elliptic operators. Part II.]{Uniform rectifiability and elliptic operators satisfying a Carleson measure 
	condition. Part II: The large constant case}

\author[S. Hofmann]{Steve Hofmann}

\address{Steve Hofmann
\\
Department of Mathematics
\\
University of Missouri
\\
Columbia, MO 65211, USA} \email{hofmanns@missouri.edu}

\author[J.M. Martell]{José María Martell}

\address{José María Martell
\\
Instituto de Ciencias Matemáticas CSIC-UAM-UC3M-UCM
\\
Consejo Superior de Investigaciones Científicas
\\
C/ Nicolás Cabrera, 13-15
\\
E-28049 Madrid, Spain} \email{chema.martell@icmat.es}

\author[S. Mayboroda]{Svitlana Mayboroda}

\address{Svitlana Mayboroda
\\
Department of Mathematics
\\
University of Minnesota
\\
Minneapolis, MN 55455, USA} \email{svitlana@math.umn.edu}

\author[T. Toro]{Tatiana Toro}

\address{Tatiana Toro 
\\ 
University of Washington 
\\
Department of Mathematics 
\\
Seattle, WA 98195-4350, USA}

\email{toro@uw.edu}

\author[Z. Zhao]{Zihui Zhao}

\address{Zihui Zhao
\\ 
Department of Mathematics
\\
University of Chicago
\\
Chicago, IL 60637, USA}

\email{zhaozh@uchicago.edu}

\thanks{The first author was partially supported by NSF grant number DMS-1664047.
The second author acknowledges that
the research leading to these results has received funding from the European Research
Council under the European Union's Seventh Framework Programme (FP7/2007-2013)/ ERC
agreement no. 615112 HAPDEGMT. He also acknowledges financial support from the Spanish Ministry of Economy and Competitiveness, through the ``Severo Ochoa Programme for Centres of Excellence in R\&D'' (SEV-2015-0554). 
The third author was partially supported by the NSF INSPIRE Award DMS 1344235, the NSF RAISE-TAQ grant DMS 1839077, and the Simons 
Foundation grant 563916, SM.
The fourth author was partially supported by the Craig McKibben \& Sarah Merner Professor in Mathematics, by NSF grant number DMS-1664867, and by the Simons Foundation Fellowship 614610. 
The fifth author was partially supported by NSF grants DMS-1361823, DMS-1500098, DMS-1664867, DMS-1902756 and by the Institute for Advanced Study. }
\thanks{This material is based upon work supported by the National Science Foundation under Grant No. DMS-1440140 while the authors were in residence at the Mathematical Sciences Research Institute in Berkeley, California, during the Spring 2017 semester.}

\date{\today}
\subjclass[2010]{35J25, 42B37, 31B35.}

\keywords{Elliptic measure, uniform domain, $A_{\infty}$ class, exterior corkscrew, rectifiability.}

\begin{abstract}
The present paper, along with its companion  \cite{HMMTZ}, establishes the correspondence between the properties of the solutions of a class of PDEs and the geometry of sets in Euclidean space. We settle the question of whether (quantitative) absolute continuity of the elliptic measure with respect to the surface measure and uniform rectifiability of the boundary are equivalent, in an optimal class of divergence form elliptic operators satisfying a suitable Carleson measure condition. The result can be viewed as a quantitative analogue of the Wiener criterion adapted to the singular $L^p$ data case.

The first step in this direction was taken in our previous paper \cite{HMMTZ}, where we considered the case in which the desired Carleson measure condition on the coefficients holds with sufficiently small constant.  In this paper we establish the final, general result, that is, 
the ``large constant case''. The key elements of our approach are a powerful extrapolation argument, which provides a general pathway to self-improve scale-invariant small constant estimates, as well as a new mechanism to transfer quantitative absolute continuity of elliptic measure between a domain and its subdomains.
\end{abstract}

\maketitle

\tableofcontents

\section{Introduction}

The present paper, together with its companion  \cite{HMMTZ} and its converse in \cite{KP} (see also \cite{DJe}) culminate many years of activity at the intersection of harmonic analysis, geometric measure theory, and PDEs, devoted to the complete understanding of necessary and sufficient conditions on the operator and the geometry of the domain guaranteeing absolute continuity of the elliptic measure with respect to the surface measure of the boundary.

The celebrated 1924 Wiener criterion \cite{Wiener} provided the necessary and sufficient conditions on the geometry of the domain responsible for the continuity of the harmonic functions at the boundary. In the probabilistic terms, it characterized the points of the boundary which are  ``seen" by the Brownian travelers coming from the interior of the domain. 

The question of finding necessary and sufficient geometric conditions which could guarantee adequate regularity, so that, roughly speaking, the pieces of the boundary are seen by the Brownian travelers according to their surface measure, turned out to be much more intricate. Curiously, already in 1916 F. \& M. Riesz correctly identified the key geometric notion in this context:  rectifiability of the boundary $\pom$, i.e., the existence of tangent planes almost everywhere 
with respect to arc length $\sigma$ on $\pom$.  In particular, they showed in  \cite{Rfm} that harmonic measure is (mutually) absolutely continuous with respect to $\sigma$ for
a simply connected domain in the plane with rectifiable boundary. It took more than a hundred years to establish the converse of the F. \& M. Riesz theorem and its higher dimensional analogues. The first such result appeared in 2016 \cite{7au}, and the question was fully settled for the harmonic functions in 2018 \cite{AHMMT}.

The question of what happens in the general PDE setting has been puzzling from the beginning.
The Wiener criterion is universal: it applies to all uniformly elliptic divergence form operators with bounded coefficients and characterizes
points of continuity of the solution at the boundary. It was realized early 
on that no such general criterion exists for determining the absolute continuity of elliptic measure with respect to the surface measure to the boundary of a domain.
Some of the challenges that arise when considering this question were highlighted by the counterexamples in \cite{CFK}, \cite{MM}. In 1984 Dahlberg 
formulated a conjecture concerning optimal conditions on a matrix of coefficients which guarantee absolute continuity of elliptic measure with respect to 
Lebesgue measure in a half-space. This question was a driving force of a thread of outstanding developments in harmonic analysis in the 80s and 90s  due to Dahlberg, Jerison, Kenig, Pipher, and others, stimulating some beautiful and far-reaching new techniques in the theory of weights and singular integral operators, to mention only a few approaches. In \cite{KP}, Kenig and Pipher proved Dahlberg's conjecture, they showed that whenever the gradient of coefficients satisfies a Carleson measure condition (to be defined below) the elliptic measure and the Lebesgue measure are mutually absolutely  continuous on a half-space and, by a change of variables argument, above a  Lipschitz graph. 

%In fact, the question of optimal conditions on a matrix of coefficients which guarantee absolute continuity of elliptic measure with respect to Lebesgue measure in a half-space has been a driving force of a thread of outstanding developments in harmonic analysis in the 80s and 90s  due to Dahlberg, Jerison, Kenig, Pipher, and others, stimulating some beautiful and far-reaching new techniques in the theory of weights and singular integral operators, to mention only a few approaches.   
%, absolute continuity of elliptic measure with respect to the surface measure cannot hold without additional assumptions on coefficients, and following the counterexamples in \cite{CFK}, \cite{MM} and the 1984 conjecture by B. Dahlberg, the optimal positive results were proved by Kenig and Pipher in \cite{KP}. Specifically, they showed that whenever the gradient of coefficients satisfies a Carleson measure condition (to be defined below) the elliptic measure and the Lebesgue measure are mutually absolutely  continuous on a half space and, by a change of variables argument, above a  Lipschitz graph. According to the aforementioned counterexamples such a Carleson measure condition is of the nature of the best possible. 
 
%thus providing the final, optimal results both geometrically and in terms of the hypotheses on the operator at hand. 

Given the aforementioned developments, it was natural to conjecture that the equivalence of rectifiability and regularity of elliptic measure should be valid in the full generality of Dahlberg-Kenig-Pipher (DKP) coefficients.
Despite numerous attempts this question turned out to be notoriously resistant to existing methods.
The passage from the regularity of the solutions to partial differential equations to rectifiability, or to any geometric information on the boundary, is generally referred to as free boundary problems. This in itself is, of course, a well-studied and rich subject. Unfortunately, the typical techniques arising from minimization of the functionals are both too qualitative and too rigid to treat structural irregularities of  rectifiable sets and such weak assumptions as absolute continuity of harmonic measure. The latter became accessible only recently, with the development of the analysis of singular integrals and similar objects on uniformly rectifiable sets. In particular, the first converse of the F. \& M. Riesz theorem, \cite{7au}, directly relies on the 2012 solution of the David-Semmes conjecture regarding the boundedness of the Riesz transforms in $L^2$ \cite{NTV}. At the same time, the techniques stemming from such results for the harmonic functions are not amenable to more general operators of the DKP type, again, due to simple yet fundamental algebraic deficiencies: the derivatives of the coefficients do not offer sufficient cancellations. In the first paper of the present sequence, \cite{HMMTZ}, the authors managed to combine ``classical" free boundary blow-up and compactness arguments (originated in geometric measure theory) with scale-invariant harmonic analysis methods to show that the desired uniform rectifiability follows from regularity of elliptic measure whenever the coefficients of the underlying equation exhibit {\it small} oscillations, in the appropriate Carleson measure sense. The smallness condition, while obviously suboptimal, could not be removed directly, for it is essentially built in the nature of the compactness arguments.

The main goal of the present paper is to addresses the conjecture in full generality. We establish
 the {\it equivalence} of the absolute continuity of the elliptic measure with respect to the surface measure and the uniform rectifiability of the boundary of a domain under the {DKP}  condition on the coefficients, thus providing the final, optimal geometric results (given the assumed background 
hypotheses).

In order to achieve this, we rely on an extrapolation argument, a powerful harmonic analysis \Bk self-improvement technique which allows one to ``bootstrap" certain quantitative scale-invariant results under a small constant hypothesis to the general, large-constant case. This requires a seamless transfer of various estimates from the initial domain to a family of special sawtooth subdomains ---one of the major technical challenges of the present proof.  

%To conclude this part of the introduction, we remark that the final resulting class of operators treated in this paper, satisfying the ``large constant" DKP condition is the best possible, as is justified by several counterexamples in Section \ref{optimal}. We will discuss them below right after the detailed statements and definitions.   

We now describe our results in more detail. Throughout the paper we shall work under the assumptions that the domain $\Omega$ is uniform, i.e., open and connected in a quantitative way, and that its boundary is $(n-1)$-Ahlfors regular, that is, $(n-1)$-dimensional in a quantitative way (see Section \ref{s3}). Under these conditions one can, for instance,  show that scale-invariant absolute continuity of harmonic  measure is related to the uniform  rectifiability of the boundary and even to the non-tangential accessibility of the exterior domain:

\begin{theorem}\label{thm:hmu}
Let $\Omega\subset\RR^n$, $n\ge 3$, be a uniform domain (bounded or unbounded) with 
Ahlfors regular boundary \textup{(}see Definitions \ref{def:uniform} and \ref{def:ADR}\textup{)}, set $\sigma=\mathcal{H}^{n-1}|_{\pO}$ 
and let $\omega_{-\Delta}$ denote its associated harmonic measure. The following statements are equivalent:
\begin{enumerate}[label=\textup{(\alph*)}, itemsep=0.2cm] 

\item\label{1-thm:hmu} $\omega_{-\Delta} \in A_\infty(\sigma)$ \textup{(}Definition \ref{def:AinftyHMU}\textup{)}.

\item\label{2-thm:hmu} $\partial\Omega$ is uniformly rectifiable \textup{(}Definition \ref{def:UR}\textup{)}.

\item\label{3-thm:hmu} $\Omega$ satisfies the exterior corkscrew condition \textup{(}see Definition \ref{def:ICC}), hence, in particular, it is a chord-arc domain \textup{(}Definition \ref{def:nta}).
\end{enumerate}
 \end{theorem}

 Postponing all the rigorous definitions to Section~\ref{s3}, we remark for the moment that uniform rectifiability is a quantitative version of the notion of rectifiability of the boundary and the Muckenhoupt condition $\omega\in A_\infty(\sigma)$ is, respectively, a quantitative form of the mutual absolute continuity of $\omega$ with respect to $\sigma$. Thus, Theorem \ref{thm:hmu} above is a quantitative form of the rigorous connection between the boundary behavior of harmonic functions and geometric properties of sets that we alluded to above. 
Returning to the ties with Wiener criterion, we point out that the property of the scale invariant absolute continuity of harmonic measure with respect to surface measure, 
at least in the presence of Ahlfors regularity of $\pom$,
is equivalent to the solvability of the Dirichlet problem
with data in some $L^p(\pom)$, with $p<\infty$\footnote{See, e.g., \cite{H}, although the result is folkloric,
	and well known in less austere settings \cite{Ke}.}; thus, such a characterization is in some sense
an analogue of Wiener's criterion for singular, rather than continuous data.

Theorem~\ref{thm:hmu} in the present form appears in \cite[Theorem 1.2]{AHMNT}. That \ref{1-thm:hmu} implies \ref{2-thm:hmu} is the main result in \cite{HMU} (see also \cite{HM4, HLMN});  that \ref{2-thm:hmu} yields \ref{3-thm:hmu} is \cite[Theorem 1.1]{AHMNT}; and the fact that \ref{3-thm:hmu}  implies \ref{1-thm:hmu} was proved in \cite{DJe}, and independently in \cite{Sem}. 
%The latter result was stated for the Laplacian, but the proof extends routinely to second order 
%elliptic operators with constant coefficients. %  (see \cite{HMT1, CHMT}).  

Theorem \ref{thm:hmu} 
and other recent results\footnote{We refer the reader also to recent work of Azzam \cite{Az}, in 
	which the author characterizes the domains with Ahlfors regular boundaries for which $\omega_{-\Delta}\in A_\infty(\sigma)$:  they are precisely the domains with uniformly rectifiable boundary which are semi-uniform in the sense of Aikawa and Hirata \cite{AH};  see also
	\cite{AHMMT, AMT-char, HM3} for related results characterizing $L^p$ solvability in the general
	case that $\omega_{-\Delta}$ need not be doubling.} illuminate how the $A_\infty$ condition\footnote{And also its
	non-doubling version, the weak $A_\infty$ condition.} 
of  harmonic measure is related to the geometry of the domain $\Omega$. Unfortunately, as we pointed out above, their proofs do not extend to the optimal class of operators with variable coefficients. Indeed, the best known results in this direction  pertain to the ``direct" rather than the ``free boundary" problem. A description of the elliptic measure in a given geometric environment, is essentially due to C. Kenig and  J. Pipher. 
In 2001  \cite{KP} C. Kenig and  J. Pipher proved what they referred to as a 1984 Dahlberg conjecture: if $\Omega\subset \RR^n$ is a bounded Lipschitz domain and the 
elliptic matrix $\mathcal{A}$ satisfies the following Carleson measure condition:%(later referred to as Kenig-Pipher condition), that is, 
\begin{equation}\label{KP-cond}
\sup_{\substack{q\in\pO \\ 0< r<\diam(\Omega)} } \frac{1}{r^{n-1}} \iint_{B(q,r) \cap\Omega} 
\bigg(
\sup_{Y\in B(X,\frac{\delta(X)}{2})} |\nabla \mathcal{A}(Y)|^2 \delta(Y)\bigg)dX <\infty,
\end{equation}
where here and elsewhere we write $\delta(\cdot)=\dist(\cdot,\partial\Omega)$, then the 
corresponding elliptic measure $\omega_L\in A_{\infty}(\sigma) $.  
As observed in \cite{HMT1}, 	
one may carry through the proof in \cite{KP}, essentially 
unchanged, with a slightly weakened reformulation of  \eqref{KP-cond}, namely
by assuming, in place of  \eqref{KP-cond}, the following properties: %\ref{H1} and \ref{H2} below. 
\begin{enumerate}[label=(H\arabic*), itemsep=0.2cm]
	\item\label{H1} $\wcalA\in \Lip_{\rm loc}(\Omega)$ and $|\nabla \wcalA|\delta(\cdot) \in L^\infty(\Omega)$, where $\delta(\cdot) := \dist(\cdot,\pO)$.
	\item\label{H2} $|\nabla \wcalA|^2 \delta(\cdot)$ satisfies the Carleson measure assumption:
	\begin{equation}\label{KP-s-relaxed}
	\|\mathcal{A}\|_{\rm Car}:=\sup_{\substack{q\in\pO\\0< r<\diam(\Omega)} } \frac{1}{r^{n-1}} \iint_{B(q,r) \cap\Omega} 
	|\nabla \mathcal{A}(X)|^2 \delta(X)dX <\infty\,.
	\end{equation}	
\end{enumerate}
We shall refer to these hypotheses (jointly) as the \textbf{Dahlberg-Kenig-Pipher  (DKP)  condition}.
Note that each of \ref{H1} and \ref{H2} is implied by \eqref{KP-cond}.

Since properties
\ref{H1} and \ref{H2} are preserved in subdomains, one can use
the method of \cite{DJe} to extend the result of \cite{KP} to chord-arc domains, 
and hence the analogue of \ref{3-thm:hmu} implies \ref{1-thm:hmu} (in Theorem \ref{thm:hmu}) 
holds for operators satisfying the DKP condition.

An attempt to address the ``free boundary" part of the problem that is to prove that  \ref{1-thm:hmu} implies  \ref{2-thm:hmu} or  \ref{3-thm:hmu} 
led, the first, second and fourth authors of the present paper (see \cite{HMT1}) to show that under the same background hypothesis as in Theorem \ref{thm:hmu},  \ref{1-thm:hmu} implies 
 \ref{3-thm:hmu} (and hence also  \ref{2-thm:hmu})
 for elliptic operators with variable-coefficient matrices $\mathcal{A}$ 
 satisfying \ref{H1} and the Carleson measure estimate 
\begin{equation}\label{HMT-CM}
\sup_{\substack{q\in\pO \\ 0< r<\diam(\Omega)} } \frac{1}{r^{n-1}} \iint_{B(q,r) \cap\Omega} 
|\nabla \mathcal{A}(X)|dX <\infty.
\end{equation}

We observe that, in the presence of hypothesis \ref{H1}, 
\eqref{HMT-CM}
implies \eqref{KP-s-relaxed}. The weighted
$W^{1,2}$ Carleson measure
estimate \eqref{KP-s-relaxed} 
 is both weaker, and more natural than the $W^{1,1}$ version \eqref{HMT-CM}.
 For example, operators verifying 
 \eqref{KP-s-relaxed} arise as pullbacks of constant coefficient operators (see \cite[Introduction]{KP}), and also
 in the linearization of ``$A$-harmonic" (i.e., generalized $p$-harmonic) operators (see \cite[Section 4]{LV}).
We also mention in
 passing that a qualitative version of the results in \cite{HMT1} was obtained in \cite{ABHM}.
There are also
related (quantitative) results in \cite{HMM} and \cite{AGMT} that are 
valid in the absence of any connectivity hypothesis.

%hence Theorem  \ref{thm:hmu} remains true for this new class of matrices. We also mention in
% passing that a qualitative version of this fact was obtained in \cite{ABHM}, 
%and that there are
%related (quantitative) results in \cite{HMM} and \cite{AGMT} that are 
%valid in the absence of any connectivity hypothesis.

%Nonetheless, it has remained an open problem to address such issues, assuming
%that the coefficients satisfy only the Dahlberg-Kenig-Pipher condition (specifically, the weighted
%$W^{1,2}$ Carleson measure
%estimate \eqref{KP-s-relaxed}, as opposed to the $W^{1,1}$ version \eqref{HMT-CM}.
%The former is both weaker, and more natural:  for example, operators verifying 
 %\eqref{KP-s-relaxed} arise as pullbacks of constant coefficient operators (see \cite[Introduction]{KP}), and also
 %in the linearization of ``$A$-harmonic" (i.e., generalized $p$-harmonic) operators (see \cite[Section 4]{LV}).

From the geometric measure theory point of view the main motivation for this paper and its companion \cite{HMMTZ} is to understand whether the  elliptic measure of a DKP divergence form elliptic 
operator distinguishes between a rectifiable and a purely unrectifiable boundary. As in Theorem \ref{thm:hmu}, we make the background assumption
that $\Omega\subset \mathbb{R}^n$, $n\ge 3$, is a uniform domain
(see Definition \ref{def:uniform}) 
 with an Ahlfors regular boundary (Definition 
\ref{def:ADR}).
Analytically we consider second order divergence form elliptic
operators, that is, 
$L=-\divg(\mathcal{A}(\cdot)\nabla)$, where $\mathcal{A}= \big( a_{ij}\big)_{i,j=1}^n$ is a (not necessarily symmetric) 
real matrix-valued function on $\Omega$, 
satisfying the usual uniform ellipticity condition 
\begin{equation}\label{def:UE}
 \langle \mathcal{A}(X)\xi, \xi  \rangle \ge \lambda |\xi|^2,
\qquad
| \langle \mathcal{A}(X)\xi, \zeta \rangle | \leq \Lambda |\xi \,|\zeta|, \qquad\text{for all } \xi,\zeta \in\mathbb{R}^n\setminus\{0\}\,,
\end{equation}
for uniform constants $0<\lambda\le \Lambda<\infty$, and for a.e. $X\in\Omega$.
We further assume that $\mathcal{A}$ 
satisfies the Dahlberg-Kenig-Pipher condition, that is, \ref{H1} and \ref{H2}, and, additionally, that the associated 
elliptic measure is an $A_\infty$ weight (see Definition \ref{def:AinftyHMU}) with 
respect to the surface measure $\sigma=\mathcal{H}^{n-1}|_{\pO}$. 
Our goal is to understand how this 
analytic information yields insight on the geometry of the
domain and its boundary. As mentioned above, in \cite{HMMTZ} we considered the case that 
the Carleson condition \ref{H2} holds with sufficiently small constants (i.e, $\|\mathcal{A}\|_{\rm Car}$ is 
small, see \eqref{KP-s-relaxed}). In this paper we prove the general result, that is, the ``large constant'' case in 
which the Carleson condition \ref{H2} is assumed merely to be finite. 
We shall utilize an extrapolation (or bootstrapping) argument to 
pass from the case of small Carleson norm to the general case.

%In \cite{HMMTZ} we considered the case that 
%the Carleson condition \ref{H2} holds with sufficiently small constants (i.e, $\|\mathcal{A}\|_{\rm Car}$ is 
%small, see \eqref{KP-s-relaxed}). In this paper we prove the optimal result, that is, the ``large constant'' case in 
%which the Carleson condition \ref{H2} is assumed merely to be finite. 

Throughout this paper, and unless otherwise specified, by \textit{allowable constants}, 
we mean the dimension $n\geq 3$; the constants involved in the definition of a uniform 
domain, that is, $M, C_1>1$ (see Definition \ref{def:uniform}); the Ahlfors regular constant 
$C_{AR}>1$ (see Definition \ref{def:ADR}); the ratio of the ellipticity constants 
$\Lambda/\lambda \geq 1$ (see \eqref{def:UE}), 
and the $A_{\infty}$ constants $C_0>1$ and 
$\theta\in(0,1)$ (see Definition 
\ref{def:AinftyHMU}).  

%\Or By renormalizing, we shall always assume that $\lambda = 1$; thus the new, renormalized
%$\Lambda$ will be equal to the ratio $\Lambda/\lambda$ of the original ellipticity constants. \Bk \chema{\Or: Do we really need the talk about normalization here? In the other paper it made sense but I believe that here we can and should avoid this}

Our main result is as follows:
\begin{theorem}\label{thm:main}
	Let $\Omega\subset \RR^n$, $n\ge 3$, be a uniform domain with Ahlfors regular boundary and set $\sigma=\mathcal{H}^{n-1}|_{\pO}$. Let $\wcalA$ be a (not necessarily symmetric) uniformly elliptic matrix on $\Omega$ satisfying \textup{\ref{H1}} and \textup{\ref{H2}}. Then the following are equivalent:
\begin{enumerate}[label=\textup{(\arabic*)}, itemsep=0.2cm] 
		\item\label{1-thm-main} The elliptic measure $\omega_L$ associated with the operator $L=-\divg(\wcalA(\cdot)\nabla)$ is of class $A_\infty$ with respect to the surface measure. 
		\item\label{3-thm-main} $\pO$ is uniformly rectifiable.
		\item\label{2-thm-main} $\Omega$ is a chord-arc domain.
		\end{enumerate} 
\end{theorem} 

%While we have already discussed the historical background, let us recuperate now that we have a precise statement. 
The equivalence of  \ref{3-thm-main} and \ref{2-thm-main} (under the stated background hypotheses)
was previously known: that \ref{2-thm-main} $\implies$ \ref{3-thm-main} follows from
the main geometric result of \cite{DJe} (namely,
that chord-arc domains can be approximated in a big pieces sense by 
Lipschitz subdomains), and the converse  \ref{3-thm-main}
$\implies$ \ref{2-thm-main}  is proved in \cite{AHMNT}. 
Moreover, as mentioned above, it was also known
that \ref{2-thm-main} $\implies$ \ref{1-thm-main}, and the proof
comprises two main ingredients:
first, that the properties \ref{H1} and \ref{H2} are preserved in subdomains, and therefore
by the result of \cite{KP}\footnote{The formulation in terms of  
 \ref{H1} and \ref{H2} in place of \eqref{KP-cond} appears in \cite{HMT1}, but the result is implicit
in \cite{KP}; see \cite[Appendix A]{HMT1}.}, $\omega_L \in A_\infty (\sigma)$ in a Lipschitz subdomains
of $\Omega$;
and second, by the aforementioned big piece approximation result of  \cite{DJe},
that the $A_\infty$ property may be passed
from Lipschitz subdomains to the original chord-arc domain, by use of the maximum principle and a 
change of pole argument (see \cite{DJe} or, originally, \cite{JK}). 
%It has been known by the work of \cite{DJe}, \cite{Se}, \cite{HMU}, and \cite{AHMNT} 
%that \ref{2-thm-main} and \ref{3-thm-main} are equivalent (under the stated background hypotheses). 
In this paper we close the circle by proving the implication \ref{1-thm-main} $\implies$ \ref{3-thm-main}, thus providing a 
%complete geometric description of a 
characterization of chord-arc domains in terms of the properties of the elliptic measure .

%\steve{\Bl This is the new version of the old paragraph that I had previously highlighted in Red.
%The new version just gives an overall sketch of ideas, without getting into details:  I 
%would prefer to avoid
%too much technical detail in the introduction.  }

The proof is based on the method of ``extrapolation of Carleson measures", by means of which we bootstrap
the small constant case treated in  \cite{HMMTZ}.  This method was first introduced in the work of 
Lewis and Murray \cite{LM}, based on a Corona type construction which has its origins
in the work of Carleson \cite{Ca}, 
and Carleson and Garnett \cite{CG}.  
In order to carry out this procedure, we shall need to transfer the $A_\infty$ property of
elliptic measure, from the original domain to sawtooth subdomains.  This last step is really the 
heart of the proof.

%\comment{Ch: I am not sure if I understand what is in orange and I think that it is confusing... the sawtooths have the quantitative properties for all the scales not for ``suitable'' scales!

%\

%\textbf{Original text:} A fundamental feature of this method is that it reduces 
%matters to verifying appropriate scale-invariant estimates (for us, 
%these will be of Carleson measure type), on certain ``sawtooth" regions. 
%\Or The latter are constructed taking smaller and smaller scales from dyadic Whitney cubes of $\Omega$. They are typically subdomains of the initial domain which  individually possess similar carefully quantified geometric properties at suitable scales, and, most importantly, are such that a controlling Carleson measure is small on them.  
%\Bk

%\

%\textbf{My sugestion:}  \Or  A fundamental feature of this method is that it reduces matters to verifying appropriate scale-invariant estimates (for us, these will be of Carleson measure type), on certain ``sawtooth" regions 
%in which one has smallness of a controlling Carleson measure  (in the present paper, the latter will be \eqref{KP-s-relaxed}). These sawtooth regions are constructed by joining appropriate Whitney subdomains associated with the dyadic cubes that lie above a certain stopping-time family and they
%inherit the quantitative geometric assumptions of the original domain  in a uniform fashion.}

 %\smallskip

We briefly discuss the organization of the paper.
In Section \ref{s3} we present analytic and geometric preliminaries. In Section \ref{S:proof-by-extrapolation} 
we first state three of the key ingredients to be used in the proof, namely:
Theorem \ref{thm:extrapolation} proved in  \cite{HMM}; Theorem \ref{thm:bhc} which comes as 
combination of \cite{HMM}, \cite{GMT}; and Theorem \ref{thm:sca} obtained in \cite{HMMTZ}. 
We then conclude Section \ref{S:proof-by-extrapolation} by outlining 
the rather intricate proof of Theorem \ref{thm:main}, and in the process
reducing matters to two main steps.  The latter are then carried out in two separate sections:
in Section \ref{s:small-extrapolation}, we prove a technical estimate showing that
a continuous parameter Carleson measure, restricted to a sawtooth subdomain, 
may be controlled quantitatively by a discretized version of itself.
Section \ref{sect:trsf} contains the most delicate technical part of the proof, involving transference of the
$A_\infty$ property to sawtooth subdomains. In Section \ref{optimal} we discuss the optimality of the results and present 
an important corollary.
% \chema{Moved what ZZ added to here} 
In particular, in Corollary \ref{cOsc} we show that Theorem \ref{thm:main} remains true when we replace the assumptions \ref{H1} and \ref{H2} by weaker assumptions involving the oscillation of the elliptic matrix in place of its gradient. \Bk

\medskip

{\bf Acknowledgments:} The authors would like to express their gratitude to Bruno Giuseppe Poggi Cevallos who pointed out that the examples in \cite{MM} could be used to access the optimality of our results. See Proposition \ref{p:bruno}. They would also like to thank MSRI for its hospitality during the Spring of 2017, all the authors were in residence there when this work was started. 
%\Bk \chema{\Or This  has been mentioned already in the footnote in the first page. Should we move our mention to Bruno to there?}

\section{Preliminaries}\label{s3}

\subsection{Definitions}
\begin{defn}\label{def:ADR}
	We say a closed set $E\subset \RR^n$ is \textbf{Ahlfors regular} with constant $C_{AR}>1$ if for any $q\in E$ and $0<r<\diam(E)$,
	\[ C_{AR}^{-1}\, r^{n-1} \leq \mathcal{H}^{n-1}(B(q,r)\cap E) \leq C_{AR}\, r^{n-1}. \]
\end{defn}

There are many equivalent characterizations of a uniformly rectifiable set, see \cite{DS2}. Since uniformly rectifiability is not the main focus of our paper, we only state one of the geometric characterizations as its definition. 
\begin{defn}\label{def:UR}
	An Ahlfors regular set $E\subset \RR^n$ is said to be \textbf{uniformly rectifiable}, if it has big pieces of Lipschitz images of $\RR^{n-1}$. That is, there exist $\widetilde\theta, \widetilde M>0$ such that for each $q\in E$ and $0<r<\diam(E)$, there is a Lipschitz mapping $\rho: B^{n-1}(0, r) \to \RR^n$ such that $\rho$ has Lipschitz norm $\leq \widetilde M$ and 
	\[ \mathcal{H}^{n-1} \left( E\cap B(q,r) \cap \rho(B^{n-1}(0,r)) \right) \geq \widetilde\theta r^{n-1}. \]
	Here $B^{n-1}(0,r)$ denote a ball of radius $r$ in $\RR^{n-1}$.
\end{defn}

\begin{defn}\label{def:ICC}
An open set $\Omega\subset\mathbb{R}^n$ is said to satisfy the \textbf{\textup{(}interior\textup{)} corkscrew condition} \textup{(}resp. the exterior corkscrew condition\textup{)} with constant $M>1$ if for every $q\in\pO$ and every $0< r<\diam(\Omega)$, there exists $A=A(q,r) \in \Omega$ \textup{(}resp. $A\in \Omega_{\rm ext}:=\mathbb{R}^n\setminus\overline{\Omega}$\textup{)} such that
 \begin{equation}\label{eqn:nta-M}
 	B\left(A, \frac{r}{M} \right) \subset B(q,r) \cap \Omega
	\qquad
	\Big(\mbox{resp. }B\left(A, \frac{r}{M} \right) \subset B(q,r) \cap \Omega_{\rm ext}.
	\Big)
	 \end{equation} 
	 The point $A$ is called  a Corkscrew point (or a non-tangential point)  relative to $\Delta(q,r)=B(q,r)\cap\pom$ in $\Omega$ (resp. $\Omega_{\rm ext}$).
\end{defn}

\begin{defn}\label{def:HCC}
An open connected set $\Omega\subset\mathbb{R}^n$ is said to satisfy the \textbf{Harnack chain condition} with constants $M, C_1>1$ if for every pair of points $A, A'\in \Omega$
there is a chain of balls $B_1, B_2, \dots, B_K\subset \Omega$ with $K \leq  M(2+\log_2^+ \Pi)$ that connects $A$ to $A'$,
where
\begin{equation}\label{cond:Lambda}
\Pi:=\frac{|A-A'|}{\min\{\delta(A), \delta(A')\}}.
\end{equation} 
Namely, $A\in B_1$, $A'\in B_K$, $B_k\cap B_{k+1}\neq\emptyset$ and for every $1\le k\le K$
\begin{equation}\label{preHarnackball}
 	C_1^{-1} \diam(B_k) \leq \dist(B_k,\partial\Omega) \leq C_1 \diam(B_k).
\end{equation}
         \end{defn}

We note that in the context of the previous definition if $\Pi\le 1$ we can trivially form the Harnack chain $B_1=B(A,3\delta(A)/5)$ and $B_2=B(A', 3\delta(A')/5)$ where \eqref{preHarnackball} holds with $C_1=3$. Hence the Harnack chain condition is non-trivial only when $\Pi> 1$.

\begin{defn}\label{def:uniform}
An open connected set $\Omega\subset\RR^n$ is said to be a \textbf{uniform} domain with constants $M, C_1$,  if it satisfies the interior corkscrew condition with constant $M$ and the Harnack chain condition with constants $M, C_1$.
\end{defn}

\begin{defn}\label{def:nta}
A uniform domain $\Omega\subset\RR^n$ is said to be \textbf{NTA} if it satisfies the exterior corkscrew condition. If one additionally assumes that $\partial\Omega$ is Ahlfors regular, the $\Omega$ is said to be a \textbf{chord-arc} domain.
\end{defn}

For any $q\in\pO$ and $r>0$, let $\Delta=\Delta(q,r)$ denote the surface ball $B(q,r) \cap \pO$, and let $T(\Delta)=B(q,r)\cap\Omega$ denote the Carleson region above $\Delta$. We always implicitly assume that $0<r< \diam ( \Omega)$. We will also write $\sigma=\mathcal{H}^{n-1}|_{\pO}$.

Given an open connected set $\Omega$ and an elliptic operator $L$ we let $\{\omega_L^X\}_{X\in\Omega}$ be the associated elliptic measure.  In the statement of our main result we assume that $\omega_L \in A_{\infty}(\sigma)$ in the following sense:

\begin{defn}\label{def:AinftyHMU}
	The elliptic measure associated with $L$ in $\Omega$ is said to be of class $A_{\infty}$ with respect to the surface measure $\sigma= \mathcal{H}^{n-1}|_{\pO}$, which we denote by $\omega_L\in A_\infty(\sigma)$, if there exist $C_0>1$ and $0<\theta<\infty$ such that for any surface ball $\Delta(q,r)=B(q,r)\cap \pO$, with $x\in\partial\Omega$ and $0<r<\diam (\Omega)$, any surface ball $\Delta'=B'\cap \pO$ centered at $\pO$ with $B'\subset B(q,r)$, and any Borel set $F\subset \Delta'$, the elliptic measure with pole at $A(q,r)$ (a corkscrew point relative to $\Delta(q,r)$) satisfies
	   \begin{equation}\label{eqn:Ainfty}
		\frac{\omega_L^{A(q,r)}(F)}{\omega_L^{A(q,r)}(\Delta')} \leq C_0\left( \frac{\sigma(F)}{\sigma(\Delta')} \right)^{\theta}.
	\end{equation}
We may refer to $(C_0,\theta)$ as the $A_\infty$ constants of $\omega_L$ with respect to $\sigma$.
\end{defn}

Since $\sigma$ is a doubling measure, it is well-known that the condition $\omega_L\in A_\infty(\sigma)$ 
is equivalent to the fact that $\omega_L\in RH_q(\sigma)$ for some $q>1$ in the following sense: $\omega_L \ll \sigma$ and the
Radon-Nikodym derivative $\textbf{k}_L:= d\omega_L/d\sigma$ satisfies the reverse H\"older estimate
\begin{equation}\label{eq1.wRH}
\left(\fint_{\Delta'} \big(\textbf{k}_L^{A(q,r)}\big)^q d\sigma \right)^{\frac1q} \lesssim \fint_{\Delta'} \textbf{k}_L^{A(q,r)} \,d\sigma\,
=\frac{\omega_L^{A(q,r)}(\Delta')}{\sigma(\Delta')}\,,
\end{equation}
for all  $\Delta(q,r)=B(q,r)\cap \pO$, with $x\in\partial\Omega$ and $0<r<\diam (\Omega)$, any surface ball $\Delta'=B'\cap \pO$ centered at $\pO$ with $B'\subset B(q,r)$.

The constants $0<\lambda\le \Lambda$ from \eqref{def:UE}, $C_{AR}$ from Definition \ref{def:ADR}, $M$, $C_1$ from Definition \ref{def:uniform} (see also Definitions \ref{def:ICC} and \ref{def:HCC}),  and $C_0$ and $\theta$ from Definition \ref{def:AinftyHMU}  are referred to as the allowable constants.

\subsection{Construction of sawtooth domains and Discrete Carleson measures}\label{section:sawtooth}

\begin{lemma}[Dyadic decomposition of Ahlfors regular set, \cite{DS1, DS2, Ch}]\label{lm:ddAR}
	Let $E\subset \RR^n$ be an Ahlfors regular set. Then there exist constants $a_0, A_1, \gamma>0$, depending only on $n$ and the constants of Ahlfors regularity, such that for each $k\in\mathbb{Z}$, there is a collection of Borel sets (``dyadic cubes'')
	\[ \mathbb{D}_k : = \{Q_j^k \subset E: j\in \mathscr{J}_k \}, \]
	where $\mathscr{J}_k$ denotes some index set depending on $k$, satisfying the following properties.
\begin{enumerate}[label= \textup{(\roman*)}, itemsep=0.2cm] 		%\begin{enumerate}[(i)]
		\item\label{1-lm:ddAR} $E = \bigcup_{j\in\mathscr{J}_k} Q_j^k$ for each $k\in\mathbb{Z}$.
		
		\item\label{2-lm:ddAR} If $m\geq k$ then either $Q_i^m \subset Q_j^k$ or $Q_i^m \cap Q_j^k = \emptyset$.
		
		\item\label{3-lm:ddAR} For each pair $(j,k)$ and each $m<k$, there is a unique $i\in\mathscr{J}_m$ such that $Q_j^k \subset Q_i^m$.
		
		\item\label{4-lm:ddAR} $\diam Q_j^k \leq A_1 2^{-k}$.
		
		\item\label{5-lm:ddAR} Each $Q_j^k$ contains some surface ball $\Delta(x_j^k, a_0 2^{-k}) := B(x_j^k, a_0 2^{-k}) \cap E$.
		
		\item\label{6-lm:ddAR} For all $(j, k)$ and all $\rho\in (0,1)$ 
		\begin{multline}\label{thin-boundary}
		\mathcal{H}^{n-1} \left( \left\{ q\in Q_j^k: \dist(q, E\setminus Q_j^k) \leq \rho 2^{-k} \right\} \right)
		\\
		 + \mathcal{H}^{n-1} \left( \left\{ q\in E\setminus Q_j^k: \dist(q, Q_j^k) \leq \rho 2^{-k} \right\} \right) \leq A_1 \rho^{\gamma} \mathcal{H}^{n-1} (Q_j^k).
\end{multline}
\end{enumerate}
	We shall denote by $\mathbb{D} = \mathbb{D}(E)$ the collection of all relevant $Q_j^k$, i.e.,
	\begin{equation}\label{all-cubes}
	\DD = \bigcup_k \DD_k, 
	\end{equation}
	where, if $\diam (E)$ is finite, the union runs 	over those $k$ such that $2^{-k} \lesssim  {\rm diam}(E)$.
\end{lemma}

\begin{remark}
	For a dyadic cube $Q\in \mathbb{D}_k$, we shall
set $\ell(Q) = 2^{-k}$, and we shall refer to this quantity as the ``length''
of $Q$.  Evidently, $\ell(Q)\approx \diam(Q).$	We will also write $x_Q$ for the ``center'' of $Q$, that is, the center of the ball appearing in \ref{5-lm:ddAR}. 
\end{remark}

Assume from now on that $\Omega$ is a uniform domain with Ahlfors regular boundary and set $\sigma=\mathcal{H}^{n-1}|_{\pO}$. Let $\dd=\dd(\pom)$ be the associated dyadic grid from the previous result. 

Let $\mathcal{W}=\W(\Omega)$ denote a collection
of (closed) dyadic Whitney cubes   of $\Omega$ (just dyadically divide the standard Whitney cubes from \cite[Chapter VI]{St} into cubes with side length 1/8 as large),  so that the boxes  in $\mathcal{W}$
form a covering of $\Omega$ with non-overlapping interiors, and  which satisfy
\begin{equation}\label{eqWh1} 4\, {\rm{diam}}\,(I)\leq \dist(4 I,\pom) \leq  \dist(I,\pom) \leq 40 \, {\rm{diam}}\,(I).
\end{equation}
Let $X(I)$ denote the center of $I$, let $\ell(I)$ denote the side length of $I$,
and write $k=k_I$ if $\ell(I) = 2^{-k}$. We will use ``boxes'' to refer to the Whitney cubes as just constructed, and ``cubes'' for the dyadic cubes on $\pO$. 
Then for each pair $I, J \in \W$,
\begin{equation}\label{eq:Iequivsize}
\text{ if } I\cap J \neq \emptyset, \text{ then } 4^{-1} \leq \frac{\ell(I)}{\ell(J)} \leq 4.
\end{equation}
Since $I, J$ are dyadic boxes, then $I\cap J$ is either contained in a face of $I$, or contained in a face of $J$. 
By choosing $\tau_0<2^{-10}$  sufficiently small (depending on $n$), we may also suppose that there is $t\in(\frac12,1)$ so that if $0<\tau<\tau_0$, for every distinct pair $I, J\in \W(\Omega)$,
\begin{equation}\label{eq:tausmall1}
(1+4\tau) I \cap (1+4\tau) J \neq \emptyset \iff I\cap J \neq \emptyset;
\end{equation}
and
\begin{equation}\label{eq:tausmall2}
	t J \cap (1+ 4\,\tau) I = \emptyset.
\end{equation}
Also, $J\cap (1+\tau )I$ contains an $(n-1)$-dimensional cube with side length of the order of $\min\{\ell(I), \ell(J)\}$. This observation will become useful in Section \ref{sect:trsf}.  For such $\tau\in (0,\tau_0)$ fixed, we write $I^*=(1+\tau)I$, $I^{**}= (1+2\tau) I$, and $I^{**}= (1+4\tau) I$ for the ``fattening'' of $I\in \W$.

Following \cite[Section 3]{HM} we next introduce the notion of \textit{\bf Carleson region} and \textit{\bf discretized sawtooth}.  Given a cube $Q\in\dd$, the \textit{\bf discretized Carleson region $\dd_{Q}$} relative to $Q$ is defined by
\[
\dd_{Q}=\{Q'\in\dd:\, \, Q'\subset Q\}.
\]
Let $\F$ be family of disjoint cubes $\{Q_{j}\}\subset\dd$. The \textit{\bf global discretized sawtooth region} relative to $\F$ is the collection of cubes $Q\in\dd$  that are not contained in any $Q_{j}\in\F$;
\[
\dd_{\F}:=\dd\setminus \bigcup\limits_{Q_{j}\in\F}\dd_{Q_{j}}.
\]
For a given $Q\in\dd$ the {\bf local discretized sawtooth region} relative to $\F$ is the collection of cubes in $\dd_{Q}$ that are not in contained in any $Q_{j}\in\F$;
\begin{equation}\label{dfq}
\dd_{\F,Q}:=\dd_{Q}\setminus \bigcup\limits_{Q_{j}\in \F} \dd_{Q_{j}}=\dd_{\F}\cap \dd_{Q}.
\end{equation}
We also introduce the ``geometric'' Carleson and sawtooth regions. For any dyadic cube $Q\in\DD$, pick two parameters $\eta\ll 1$ and $K\gg 1$, and define
\begin{equation}\label{def:WQ0}
	\WW_Q^0 : = \{I\in\WW: \eta^{\frac{1}{4}} \ell(Q) \leq \ell(I) \leq K^{\frac{1}{2}} \ell(Q),\  \dist(I,Q) \leq K^{\frac{1}{2}} \ell(Q)\}.
\end{equation} 
Taking $K\ge 40^2 n$, if $I\in\W$ and we pick $Q_I\in \dd$ so that $\ell(Q_I)=\ell(I)$ and $\dist(I,\pom)=\dist(I,Q_I)$, then $I\in \W^0_{Q_I}$. 
Let $X_Q$ denote a corkscrew point for the surface ball $\Delta(x_Q, r_Q/2)$. We can guarantee that $X_Q $ is in some $I\in\WW_Q^0$ provided we choose $\eta$ small enough and $K$ large enough. For each $I\in\WW_Q^0$, there is a Harnack chain connecting $X(I)$ to $X_Q$, we call it $\mathcal{H}_I$. By the definition of $\WW_Q^0$ we may construct this Harnack chain so that it consists of a bounded number of balls (depending on the values of $\eta, K$). We let $\WW_Q$ denote the set of all $J\in\WW$ which meet at least one of the Harnack chains $\mathcal{H}_I$, with $I\in \WW_Q^0$, i.e.
\begin{equation}
	\WW_Q := \{J\in\WW: \text{there exists } I\in\WW_Q^0 \text{ for which } \mathcal{H}_I \cap J \neq \emptyset\}.
\end{equation}
Clearly $\WW_Q^0 \subset \WW_Q$. Besides, it follows from the construction of the augmented collections $\WW_Q$ and the properties of the Harnack chains that there are uniform constants $c$ and $C$ such that
\begin{equation}\label{def:WQ}
	c\eta^{\frac12 } \ell(Q) \leq \ell(I) \leq CK^{\frac{1}{2}} \ell(Q), \quad \dist(I,Q) \leq CK^{\frac{1}{2}} \ell(Q) 
\end{equation}
for any $I\in \WW_Q$. In particular once $\eta, K$ are fixed, for any $Q\in \DD$ the cardinality of $\WW_Q$ is uniformly bounded.
Finally, for every $Q$ we define its associated Whitney region
\begin{equation}\label{eq2.whitney3}
U_Q := \bigcup_{I\in\,\mathcal{W}_Q} I^*.
\end{equation}

We refer the reader to \cite[Section 3]{HM} or \cite[Section 2]{HMM2} for additional details.

For a given $Q\in\dd$, the {\bf Carleson box} relative to $Q$ is defined by
\begin{equation}\label{carleson-box}
T_{Q}:=\mbox{int}\left(\bigcup\limits_{Q'\in\dd_{Q}} U_{Q'}\right).
\end{equation}
For a given family $\F$ of disjoint cubes $\{Q_{j}\}\subset\dd$ and a given $Q\in\dd$ we define the {\bf local sawtooth region} relative to $\F$ by
\[
\Omega_{\F,Q}:=\mbox{int}\left(\bigcup\limits_{Q'\in\dd_{\F,Q}}U_{Q'}\right)
=
{\rm int }\,\left(\bigcup_{I\in\,\W_{\F,Q}} I^*\right),
\]
where $\W_{\F,Q}:=\bigcup_{Q'\in\dd_{\F,Q}}\W^*_{Q'}$.
Analogously, we can slightly fatten the Whitney boxes and use $I^{**}$ to define new fattened Whitney regions and sawtooth domains. More precisely,
\begin{equation}
T_{Q}^*:=\mbox{int}\left(\bigcup\limits_{Q'\in\dd_{Q}} U_{Q'}^*\right),
\qquad
\Omega_{\F,Q}^*:=\mbox{int}\left(\bigcup\limits_{Q'\in\dd_{\F,Q}}U_{Q'}^*\right)
,
\qquad
U_{Q'}^*:=\bigcup_{I\in\,\mathcal{W}^*_{Q'}} I^{**}.
\label{eq:fatten-objects}
\end{equation}
Similarly, we can define $T_Q^{**}$, $\Omega_{\F,Q}^{**}$ and $U_{Q}^{**}$ by using $I^{***}$ in place of $I^{**}$.

One can easily see that there is a constant $\kappa_0>0$ (depending only on the allowable parameters, $\eta$, and $K$)  so that 
\begin{equation}\label{tent-Q}
T_Q\subset T_{Q}^*\subset T_{Q}^{**}\subset\overline{T_Q^{**}}\subset \kappa_0 B_Q\cap\overline{\Omega}
=: B_Q^*\cap \overline{\Omega}, \qquad\forall\,Q\in\dd.
\end{equation}

Given a pairwise disjoint family $\F\subset\dd$ (we also allow $\F$ to be the null set) and a constant $\rho>0$, we derive another family $\F({\rho})\subset\dd$  from $\F$ as follows. Augment $\F$ by adding cubes $Q\in\dd$ whose side length $\ell(Q)\leq \rho$ and let $\F(\rho)$ denote the corresponding collection of maximal cubes with respect to the inclusion. Note that the corresponding discrete sawtooth region $\dd_{\F(\rho)}$ is the union of all cubes $Q\in\dd_{\F}$ such that $\ell(Q)>\rho$.
For a given constant $\rho$ and a cube $Q\in \dd$, let $\dd_{\F(\rho),Q}$ denote the local discrete sawtooth region and let $\Omega_{\F(\rho),Q}$ denote the local geometric sawtooth region relative to disjoint family $\F(\rho)$.

Given $Q\in\dd$ and $0<\epsilon<1$, if we take $\F_0=\emptyset$, one has that \
$\F_0(\epsilon \,\ell(Q))$ is the collection of $Q'\in \dd$ such that $\epsilon\,\ell(Q)/2<\ell(Q')\le \epsilon\,\ell(Q)$. We then introduce $U_{Q,\epsilon}=\Omega_{\F_0(\epsilon \,\ell(Q)),Q}$, which is a Whitney region relative to $Q$ whose distance to $\pom$ is of the order of $\epsilon\,\ell(Q)$.  For later use, we observe that given $Q_0\in\dd$, the sets $\{U_{Q,\epsilon}\}_{Q\in\dd_{Q_0}}$ have bounded overlap with constant that may depend on $\epsilon$. Indeed, suppose that there is $X\in U_{Q,\epsilon}\cap U_{Q',\epsilon}$ with $Q$, $Q'\in \dd_{Q_0}$. By construction $\ell(Q)\approx_\epsilon \delta(X)\approx_\epsilon \ell(Q')$ and
$$
\dist(Q,Q')\le \dist(X,Q)+\dist(X,Q')\lesssim_\epsilon \ell(Q)+\ell(Q')\approx_\epsilon \ell(Q).
$$
The bounded overlap property follows then at once.

\begin{lemma}[{\cite[Lemma 3.61]{HM}}]\label{lemma:sawtooth-inherit}
	Let $\Omega\subset \RR^n$ be a uniform domain with Ahlfors regular boundary. Then all of its Carleson boxes $T_Q$ and sawtooth domains $\Omega_{\F, Q}$, $\Omega_{\F, Q}^*$  are uniform domains with Ahlfors regular boundaries. In all the cases the implicit constants are uniform, and depend only on
	dimension and on the corresponding constants for $\Omega$. 
\end{lemma}

We say that  $\PP$ is a \textbf{fundamental chord-arc subdomain} of $\Omega$ if there is $I\in \W$ and $m_1$ such that 
\begin{equation}\label{fundamental-casd}
\PP=\interior\left( \bigcup_{j=1}^{m_1} I_j^\ast\right) \hbox{ where } I_j\in\W \hbox{  and  } I\cap I_j \not =\emptyset.
\end{equation}
Note that the fact that $I\cap I_j \not =\emptyset$ ensures that $\ell(I)\approx \ell(I_j)$.  Moreover $\PP$ is a chord-arc domain with constants that only depend on $n$, $\tau$ and the constants used in the  construction of $\DD$ and $\W$ (see \cite[Lemma 2.47]{HMU} for a similar argument).

Given a sequence of non-negative numbers 
$\alpha=\{\alpha_Q\}_{Q\in\dd}$ we define the associated discrete ``measure'' $m=m_\alpha$: 
\begin{equation}\label{def:mQ}
	m(\DD') := \sum_{Q\in \DD'} \alpha_Q,
	\qquad \dd'\subset\dd
\end{equation}

\begin{defn}\label{discrete-carleson}
 Let $E\subset \RR^n$ be an Ahlfors regular set, and let  $\sigma$ be a dyadically doubling Borel measure on $E$ (not necessarily equal to $\mathcal{H}^{n-1}|_{\pO}$). We say that $m$ as defined in \eqref{def:mQ} is a discrete Carleson measure with respect to $\sigma$, if 
	\begin{equation}\label{mCarleson}
		\|m\|_{\C}: = \sup_{Q\in \DD} \frac{ m(\DD_Q)}{ \sigma(Q) } <\infty.
	\end{equation}
Also, fixed $Q_0\in\dd$ we say that $m$ is a discrete Carleson measure with respect to $\sigma$ in $Q_0$ if 
\begin{equation}\label{mCarleson-local}
\|m\|_{\C(Q_0)}: = \sup_{Q\in \DD_{Q_0}} \frac{ m(\DD_Q)}{ \sigma(Q) } <\infty.
\end{equation}
\end{defn}

\subsection{Properties of solutions and elliptic measures}

For following lemmas, we always assume that $\mathcal{D}$ is a uniform domain with Ahlfors regular boundary and $\Delta(x,r)$ denotes the surface ball $B(x,r)\cap \partial\mathcal{D}$ centered at $x\in\partial\mathcal{D}$. Let $L=-\divg(\wcalA(\cdot)\nabla)$ be a real uniformly elliptic operator, and we write $\omega=\omega_L$ for the corresponding elliptic measure. Although in our main result we consider non necessarily symmetric uniformly elliptic matrices, we will reduce matters to the symmetric case, in particular all the following properties will be used in that case, hence during this section we assume that $\wcalA$ is symmetric. 
All constants will only depend on the allowable constants, that is, those involved in the fact that the domain in question is uniform and has Ahlfors regularity boundary, and also in the uniform ellipticity of $\wcalA$. In later sections we will apply these lemmas to $\Omega$ as well as its sawtooth domains $\Omega_{\F,Q}$. For a comprehensive treatment of the subject and the proofs we refer the reader to the forthcoming monograph \cite{HMT2} (see also \cite{Ke} for the case of NTA domains). 

\begin{lemma}[Comparison principle]\label{lm:comp}
	Let $u$ and $v$ be non-negative solutions to $Lu=Lv=0$ in $B(x,4r)\cap\mathcal{D}$ which vanish continuously on $\Delta(x,4r)$.  Let $A=A(x,r)$ be a corkscrew point relative to $\Delta(x,r)$. Then
	\begin{equation}\label{P6}
		\frac{u(X)}{v(X)} \approx \frac{u(A)}{v(A)} \quad\text{~for any~} X\in B(x,r)\cap\mathcal{D}. 
	\end{equation} 
\end{lemma}

\medskip
 
\begin{lemma}[Non-degeneracy of elliptic measure]\label{lm:Bourgain}
	There exists a constant $C >1$ such that for any $x\in\partial\mathcal{D}$, $0<r<\diam(\partial\mathcal{D})$, we have
	\begin{equation}\label{eq:nondegclose}
		\omega^X(B(x,r)\cap \partial\mathcal{D}) \geq C^{-1} \qquad \text{for } X\in B(x,r/2)\cap \mathcal{D}.
	\end{equation}
\end{lemma}

\medskip

\begin{lemma}[Change of pole formula]\label{lm:cop}
	Let $x\in\partial\mathcal{D}$ and $0<r<\diam(\partial\mathcal{D})$ be given, and let $A = A(x,r)$ be a corkscrew point relative to $\Delta(x,r)$. Let $F, F' \subset \Delta(x,r)$ be two Borel subsets such that $\omega^A(F)$ and $\omega^A(F')$ are positive. Then 
	\begin{equation}
		\frac{\omega^X(F)}{\omega^X(F')} \approx \frac{\omega^A(F)}{\omega^A(F')}, \quad \text{ for any }X\in \mathcal{D} \setminus B(x,2r).
	\end{equation}
	In particular with the choice $F=\Delta(x,r)$, we have
	\begin{equation}
		\frac{\omega^X(F')}{\omega^X(\Delta(x,r))} \approx \omega^A(F') \quad \text{ for any } X\in\mathcal{D} \setminus B(x,2r).
	\end{equation}
\end{lemma}

\medskip

\begin{lemma}[CFMS estimate]\label{lm:CFMS}
	There exists a constant $C\geq 1$, such that for any $x\in\partial\mathcal{D}$, $0<r<\diam(\partial\mathcal{D})$, and $A = A(x,r)$, a corkscrew point relative to $\Delta(x,r)$, the Green's function $G=G_L$ satisfies
	\begin{equation}\label{eq:CFMS}
		C^{-1} \frac{G(X_0,A)}{r} \leq \frac{\omega^{X_0}(\Delta(x,r))}{r^{n-1}} \leq C \frac{G(X_0,A)}{r}
	\end{equation}
	for any $X_0 \in \mathcal{D} \setminus B(x,2r)$.
\end{lemma}

\medskip

\begin{lemma}[Doubling property of the elliptic measure]\label{lm:doubling}
	For every $x\in\partial\mathcal{D}$ and  $0<r<\diam(\partial\mathcal{D})$,  we have
	\begin{equation}\label{eq:doubling}
		\omega^X(\Delta(x,2r)) \leq C \omega^X(\Delta(x,r))
	\end{equation}
	for any $X\in\mathcal{D} \setminus B(x,4r)$.
\end{lemma}

\medskip

\begin{corollary}[Doubling property of the kernel]\label{cor:doublingPk}
		Let $Q\in \DD(\partial\mathcal{D})$ be a dyadic cube, and $\widetilde{Q}\in\dd$ be such that $C_1^{-1}\ell(Q)\le \ell(\widetilde{Q})\le C_1\ell(Q)$ and $\dist(Q,\widetilde{Q})\le C_1\ell(Q)$ for some $C_1\ge 1$. Suppose $\omega\in RH_p(\sigma)$ for some $p>1$, then for $X_Q$ the corkscrew relative to $Q$ we have 
		\begin{equation}
			\int_{\widetilde{Q}} \left( \textup{\textbf{k}}^{X_Q} \right)^p d\sigma \leq C \int_Q \left( \textup{\textbf{k}}^{X_Q} \right)^p d\sigma,
		\end{equation}
		with a constant $C$ depending on $C_1$ and the allowable constants and where $\textup{\textbf{k}}=d\omega/d\sigma$.
\end{corollary}
The proof is a simple corollary of the doubling property of the elliptic measure:
		\begin{align*}
			\Big(\fint_{\widetilde Q} \left(\textbf{k}^{X_Q} \right)^p d\sigma\Big)^\frac1p
			& \lesssim 
			\frac{\omega^{X_Q}(\widetilde Q)}{\sigma(\widetilde Q)}
			\approx 
			\frac{\omega^{X_Q}( Q)}{\sigma( Q)}
			=
			\fint_{Q} \textbf{k}^{X_Q} d\sigma 
			\le
			\Big(\fint_{Q} \left(\textbf{k}^{X_Q} \right)^p d\sigma\Big)^\frac1p. 
		\end{align*}

\section{Proof by extrapolation}\label{S:proof-by-extrapolation}

In this section we present some powerful tools which will be key in the proof of our main result. After that we will describe how to apply those results in our context.

We start with \cite[Lemma 4.5]{HMM}, an extrapolation for Carleson measure result which in a nutshell describes how the relationship between a discrete Carleson measure
$m$ and another discrete measure $\widetilde m$ yields information about $\widetilde m$.

\begin{theorem}[{Extrapolation, \cite[Lemma 4.5]{HMM}}]\label{thm:extrapolation}
	Let $\sigma$ be a dyadically doubling Borel measure on $\pO$ (not necessarily equal to $\mathcal{H}^{n-1}|_{\pO}$), and let $m$ be a discrete Carleson measure 
with respect to $\sigma$ (defined as in \eqref{def:mQ} and Definition \ref{discrete-carleson}),
with constant $M_0$, that is
		\begin{equation}\label{m:extrapol}
		\|m\|_{\C}: = \sup_{Q\in \DD} \frac{ m(\DD_Q)}{ \sigma(Q) } \leq M_0.
	\end{equation}
	Let $\widetilde m$ be another discrete non-negative measure on $\DD$ defined as in \eqref{def:mQ}, by
	\begin{equation*}
		\widetilde m(\DD') := \sum_{Q\in \DD'} \beta_Q, \qquad \DD'\subset\DD. 
		\end{equation*}
Assume there is a constant $M_1$ such that
\begin{equation}\label{eq:betaQ}
\quad 0\le \beta_Q \le M_1 \sigma(Q)\  \hbox{    for any   }Q\in  \DD
\end{equation}
and that there is a positive constant $\gamma$ such that for every $Q\in \DD$ and every family of pairwise disjoint dyadic subcubes $\F=\{Q_j\}\subset \DD_Q$ verifying
	\begin{equation}\label{H:small}
		\|m_\F\|_{\C(Q)}:= \sup_{Q'\in \DD_Q} \frac{m \left( \DD_{\F,Q'}  \right) }{\sigma(Q')} \leq \gamma,
	\end{equation}
	we have that $\widetilde m$ satisfies
	\begin{equation}\label{H:bd}
		\widetilde m (\DD_{\F,Q}) \leq M_1 \sigma(Q).
	\end{equation}
	Then $\widetilde m$ is a discrete Carleson measure, with
	\begin{equation}
		\| \widetilde m\|_{\C}:= \sup_{Q\in \DD} \frac{\widetilde m(\DD_Q)}{\sigma(Q)} \leq M_2,
	\end{equation}
	for some $M_2 <\infty$ depending on $n, M_0, M_1, \gamma$ and the doubling constant of $\sigma$.
\end{theorem}

\begin{theorem}[\cite{HMM}, \cite{GMT}]\label{thm:bhc}
	Let $\mathcal{D}$ be an open set satisfying an interior corkscrew condition  with Ahlfors regular boundary. Then the following are equivalent:
\begin{enumerate}[label=\textup{(\alph*)}, itemsep=0.2cm] 	
		\item\label{1-thm:bhc} $\partial\mathcal{D}$ is uniformly rectifiable.
			
		\item\label{2-thm:bhc} There exists a constant $C$ such that for every bounded harmonic function  $u$ in $\mathcal{D}$, i.e. $-\Delta u = 0$ in $\mathcal{D}$, and for any $x\in \partial\mathcal{D}$ and $0<r\lesssim \diam(\mathcal{D})$, there hold
			\begin{equation}\label{eq:bhc}
				\frac{1}{r^{n-1}} \iint_{B(x,r) \cap \mathcal{D}} |\nabla u(Y)|^2 \dist(Y, \partial\mathcal{D})\, dY \leq C \|u\|_\infty^2.
			\end{equation}
	\end{enumerate}			
\end{theorem}

\begin{remark}
	Condition \eqref{eq:bhc} is sometimes referred to as the \textit{Carleson measure estimate} (CME)
	for bounded, harmonic functions.	
\end{remark}

The direction \ref{1-thm:bhc} $\implies$ \ref{2-thm:bhc} is proved by the first three authors of the present paper \cite[Theorem 1.1]{HMM}, and the converse direction is proved by Garnett, Mourgoglou, and Tolsa \cite[Theorem 1.1]{GMT}.  As we have noted above, see Theorem \ref{thm:hmu}, under the uniform domain assumption, the statements \ref{1-thm:bhc} and/or \ref{2-thm:bhc} are equivalent to the fact that $\mathcal{D}$ is a chord-arc domain.

\begin{theorem}[{\cite[Main Theorem]{HMMTZ}}]\label{thm:sca} 
	Given the values of allowable constants $M$, $C_1$, $C_{AR}>1$, $\Lambda \geq \lambda >0$, $C_0>1$, and $0<\theta<1$, there exists $\epsilon>0$ depending on the dimension $n$ and the allowable constants, such that the following holds.
	Let $\mathcal{D} \subset \RR^n$ be a \textbf{bounded} uniform domain with constants $M, C_1$ and whose boundary is Ahlfors regular with constant $C_{AR}$. Let $\wcalA\in \Lip_{\rm loc}
(\mathcal{D})$ be a \textbf{symmetric} elliptic matrix satisfying \eqref{def:UE} with ellipticity constants $\lambda, \Lambda$, such that
\begin{enumerate}[label=\textup{(\alph*)}, itemsep=0.2cm] 	
		\item\label{1-thm:sca}  $|\nabla \wcalA|^2 \dist(\cdot, \partial\mathcal{D})$ satisfies the Carleson measure assumption with norm bounded by $\epsilon$, that is,
		\begin{equation}\label{carleson-tt-alt}
		\sup_{\substack{x\in\pO\\0< r<\diam(\Omega)} } \frac{1}{r^{n-1}} \iint_{B(x,r) \cap\Omega} 
		|\nabla \mathcal{A}(Y)|^2 \dist(Y, \partial\mathcal{D})\,dY <\epsilon\,.
		\end{equation}

		\item\label{2-thm:sca} The elliptic measure $\omega_L$ associated with the operator $L=-\divg(\wcalA\nabla)$ is of class $A_\infty$, with constants $C_0$ and $\theta$.
	\end{enumerate} 
	Then $\mathcal{D}$ is a chord-arc domain.
\end{theorem}

The proof of Theorem \ref{thm:main} is rather involved thus we sketch below the plan of the proof. 

\begin{proof}[Proof of Theorem \ref{thm:main}]
We first reduce matters to the case on which $A$ is symmetric. To do so we observe that by \cite[Theorem 1.6]{CHMT}, under the assumptions \ref{H1} and \ref{H2}, if $\omega_L\in A_\infty(\sigma)$ then $\omega_{L^{\rm sym}}\in A_\infty(\sigma)$ where $L^{\rm sym}=-\divg(\wcalA^{\rm sym}\nabla)$ and $\wcalA^{\rm sym}=(\frac{a_{ij}+a_{ji}}2)_{i,j=1}^n$ is the symmetric part of $A$. Note that, clearly, $\wcalA^{\rm sym}$ is a symmetric uniformly elliptic matrix in $\Omega$ with the same ellipticity constants as $\wcalA$. It also satisfies  \ref{H1} and \ref{H2} with constants which are controlled by those of $\wcalA$. Hence we only need to show 
\ref{1-thm-main} $\implies$ \ref{3-thm-main} for $L^{\rm sym}$ which is associated to the symmetric matrix $\wcalA^{\rm sym}$.  That is, we may assume to begin with, and we do so, that $\wcalA$ is symmetric.

Our main goal is to use the above extrapolation theorem with $m$ and $\widetilde{m}$ two discrete measures associated respectively with the sequences $\alpha=\{\alpha_Q\}_{Q\in\dd}$ and $\beta=\{\beta_Q\}_{Q\in\dd}$ defined by
\begin{equation}\label{def-alpha-beta}
\alpha_Q: = \iint_{U_Q^\ast} |\nabla \wcalA(Y)|^2 \delta(Y) dY,\qquad 
\beta_Q ; = \iint_{U_Q} |\nabla u(Y)|^2 \delta(Y) dY,
\end{equation}
where $u$ is an arbitrary bounded, harmonic function in $\Omega$, such that $\|u\|_{L^\infty(\Omega)} \leq 1$; and $U_Q$ and $U_Q^\ast$ are as defined in \eqref{eq2.whitney3} and \eqref{eq:fatten-objects} respectively.
We would like to observe that by the interior Caccioppoli inequality, $\beta_Q$ clearly satisfies the assumption \eqref{eq:betaQ}:
\begin{multline*}
\beta_Q 
= 
\iint_{U_Q} |\nabla u(Y)|^2 \delta(Y)\, dY
\lesssim
\sum_{I\in \W_Q} \ell(I)\,\iint_{I^*} |\nabla u(Y)|^2\,dY
\lesssim
\sum_{I\in \W_Q} \ell(I)^{-1}\,\iint_{I^{**}} |u(Y)|^2\,dY
\\
\lesssim
\ell(Q)^{-1}\,\iint_{U_Q^*} |u(Y)|^2 \delta(Y) dY
\le
\ell(Q)^n
\approx \sigma(Q),
\end{multline*}
where we have used that \eqref{def:WQ}, the bounded overlap of the family $\{I^{**}\}_{I\in\W}$, and \eqref{tent-Q}. We will take any family of pairwise disjoint dyadic subcubes $\F=\{Q_j\}\subset \DD_Q$ so that \eqref{H:small} holds for sufficiently small $\gamma\in (0,1)$ to be chosen and the goal is to obtain \eqref{H:bd}. To achieve this  we will carry out the following steps:

\begin{enumerate}[label=\textup{Step \arabic*}:, ref=\textup{Step \arabic*}, itemsep=0.2cm]

\item\label{1-proof-main} We first observe that \eqref{m:extrapol} is equivalent to the Carleson measure assumption \ref{H2}. This is a simple calculation which uses the fact that the Whitney boxes  $I^{\ast\ast}$ which form $U_Q^\ast$ have finite overlap and the definition of $T_Q$ in \eqref{carleson-box}, details are left to the reader.

\item\label{2-proof-main} Given $\epsilon>0$ we verify that the small Carleson hypothesis \eqref{H:small} implies that if $\gamma=\gamma(\epsilon)$ is small enough $\wcalA$ satisfies the small Carleson assumption in the sawtooth domain $\Omega_{\F,Q}^\ast$, that is, \eqref{carleson-tt-alt} holds with $\mathcal{D}=\Omega_{\F,Q}^\ast$ and the given $\epsilon$. This is done in Section \ref{s:small-extrapolation}.

\item\label{3-proof-main} We verify that under the hypotheses \ref{H1} and \ref{H2}, the assumption $\omega \in A_\infty(\sigma)$ in $\Omega$ is transferable to any sawtooth domain, in particular, if we write $\omega_{*}$ for the elliptic measure associated with $L$ in $\Omega_{\F,Q}^\ast$ then  $\omega_{*}\in A_\infty(\HH^{n-1}|_{\partial{\Omega_{\F,Q}^\ast}})$ and the implicit constants are uniformly controlled by the allowable constants. See Theorem \ref{thm:trsf} and Corollary \ref{corol:KP-transfer}.

\item\label{4-proof-main}  We combine \ref{2-proof-main} and \ref{3-proof-main}  with Theorem \ref{thm:sca} applied to $\mathcal{D}=\Omega_{\F,Q}^\ast$ and obtain that $\Omega_{\F,Q}^\ast$ is a chord-arc domain. More precisely, note first that $\Omega_{\F,Q}^\ast$ is a \textbf{bounded} uniform domain with Ahlfors regular boundary (see Lemma \ref{lemma:sawtooth-inherit}) and all the implicit constants are uniformly controlled by those of $\Omega$, that is, they do not depend on $Q$ or the family $\F$. Also, \ref{3-proof-main} says that $\omega_{*}\in A_\infty(\HH^{n-1}|_{\partial{\Omega_{\F,Q}^\ast}})$ and the implicit constants are uniformly controlled by the allowable constants. Hence for the parameter $\epsilon$ given by Theorem \ref{thm:sca} (recall that we have assumed that $\wcalA$ is \textbf{symmetric}), which only depends on the allowable constants and is independent of $Q$ or the family $\F$, we can find the corresponding $\gamma=\gamma(\epsilon)$ from \ref{2-proof-main} so that \eqref{carleson-tt-alt} holds with $\mathcal{D}=\Omega_{\F,Q}^\ast$ and that value of $\epsilon$. Thus Theorem \ref{thm:sca} applied to $\mathcal{D}=\Omega_{\F,Q}^\ast$ yields that $\Omega_{\F,Q}^\ast$ is a chord-arc domain with constants that only depend on the allowable constants.

\item\label{5-proof-main} We next apply Theorem \ref{thm:bhc} with $\mathcal{D}=\Omega_{\F,Q}^\ast$ to obtain that \eqref{eq:bhc} holds with $\mathcal{D}=\Omega_{\F,Q}^\ast$. Seeing that the latter implies 	\eqref{H:bd} is not difficult. Indeed, note that any $Y \in \Omega_{\F,Q}$ satisfies $\delta_*(Y):= \dist(Y, \partial\Omega_{\F,Q}^\ast ) \approx_\tau \delta(Y)$ (here we would like to remind the reader that $\Omega_{\F,Q}$ is comprised of fattened Whiney boxes $I^*=(1+\tau)I$ while for $\Omega_{\F,Q}^*$ we use the fatter versions $I^{**}=(1+2\tau)I$).  Thus by \eqref{eq:bhc}, the fact that $u$ is harmonic and bounded by $1$ in $\Omega$, and so in $\Omega_{\F,Q}^*$, and a simple covering argument, we can conclude that
\begin{multline*}
\null\hskip1.2cm	\widetilde m_{}( \DD_{\F,Q}) \lesssim \iint_{\Omega_{\F,Q}} |\nabla u|^2 \delta(Y) dY  \approx \iint_{\Omega_{\F,Q}} |\nabla u|^2 \delta_* (Y) dY \\
\lesssim	 \iint_{\Omega_{\F,Q}^\ast}  |\nabla u|^2 \delta_* (Y) dY  \lesssim \diam(\Omega_{\F,Q}^\ast)^{n-1} \approx \ell(Q)^{n-1} \approx \sigma(Q),
\end{multline*}
which is \eqref{H:bd}.
\end{enumerate} 

\medskip

After all these steps have been carried out the extrapolation for Carleson measures in Theorem \ref{thm:extrapolation} allows us to conclude that $\widetilde m$ is a discrete Carleson measure. In other words, we have proved that any bounded harmonic function in $\Omega$ satisfies \eqref{eq:bhc} with $\mathcal{D}=\Omega$. As a result, and by another use of Theorem \ref{thm:bhc} this time with $\mathcal{D}=\Omega$, we derive that $\partial\Omega$ is uniformly rectifiable. This completes the proof of Theorem \ref{thm:main} modulo establishing \ref{2-proof-main} and \ref{3-proof-main} and this will be done in the following sections.
\end{proof}

\begin{remark}
For convenience, we augment $\F$ by adding all subcubes of $Q$ of length $2^{-N}\ell(Q)$, 
and  let $\F_N$ denote the maximal cubes
in the resulting augmented collection. Note that for each $N\ge 2$, the sawtooth domain $\Omega_{\F_N,Q}$ is compactly contained in $\Omega$ (indeed is $2^{-N}\ell(Q)$-away from $\pom$). Note that $\DD_{\F_{N},Q} \subset \DD_{\F_{N'},Q} \subset \DD_{\F,Q} $ for every $2\le N\le N'$. In particular, $m_{\F_{N}} \leq m_{\F_{N'}} \leq m_{\F}$ and thus
\[ m_{\F} \text{ satisfies  \eqref{H:small}} \implies m_{\F_N} \text{ also satisfies the \eqref{H:small} with a constant independent of }N. \] 
We are going to prove \ref{2-proof-main} and \ref{3-proof-main} for the sawtooth domain $\Omega_{\F_N,Q}$, with constants independent of $N$. Then by  
\ref{4-proof-main} and \ref{5-proof-main}, we will have 
\begin{equation}
	\widetilde m(\DD_{\F_N,Q}) \leq M_1 \sigma(Q),
\end{equation}
with a constant $M_1$ independent of $N$, and thus \eqref{H:bd} follows from monotone convergence theorem by letting $N\to \infty$. To simplify the notations we drop from  the index $N$ from now on and write $\F=\F_N$ but we keep in mind that the corresponding sawtooth domain $\Omega_{\F,Q}$ is compactly contained in $\Omega$.
\end{remark}

\section{Consequences of the small Carleson hypothesis in the extrapolation theorem.}\label{s:small-extrapolation}

Set $\Omega_*:= \Omega_{\F,Q}^\ast$ and let $\epsilon$ be given. The goal is to see that we can find $\gamma=\gamma(\epsilon)\in (0,1)$ so that  \eqref{H:small} 
implies
\begin{equation}\label{sC:alt}
\iint_{B(x,r) \cap \Omega_*} |\nabla \wcalA(Y)|^2 \delta_*(Y) dY \leq \epsilon r^{n-1},
\end{equation}
for any $x\in\pom_*$ and any $0<r<\diam(\pom_*)$. To see this we fix $x\in\pom_*$ and $0<r<\diam(\pom_*)\approx\ell(Q)$. Using that $\Omega_* \subset \Omega$ one has that  $\delta_*(Y) \leq \delta(Y)$ and therefore \eqref{sC:alt} follows at once from 
\begin{equation}\label{sC}
\iint_{B(x,r) \cap \Omega_*} |\nabla \wcalA(Y)|^2 \delta(Y) dY \leq \epsilon r^{n-1}.
\end{equation}

To show \eqref{sC}, we let $c\in (0,1)$ be a small constant and $\widetilde{M}\ge 1$ be a large constant to be determined later, depending on the values of $\eta, K$ used in the definition of $\WW_Q^0 $ in \eqref{def:WQ0}. We consider two cases depending on the size of $r$ with respect to $\delta(x)$ for $x\in\pom^\ast$. Recall that $\Omega_*$ is compactly contained in $\Omega$, thus $\delta(x)>0$ for any $x\in \pO_*$.

\noindent
\textbf{Case 1}. $r\leq c\delta(x)$. Since $x\in\pom_*= \partial\Omega_{\F,Q}^\ast$ there exist $Q_x\in \DD_{\F,Q}$ and $I_x\in \W_{Q_x}$ such that $x\in \partial I_x^{**}$.  We choose and fix $c$ sufficiently small (depending just on dimension), so that $B(x,r)$ is contained in $2I_x$. We consider two sub-cases. First if $r\le \gamma^\frac1n\delta(x)$ then we can invoke \ref{H1} to obtain 
\begin{multline}\label{Case1a}
\iint_{B(x,r) \cap \Omega_*} |\nabla \wcalA(Y)|^2 \delta(Y) dY
\le
\iint_{B(x,r) \cap 2I_x} |\nabla \wcalA(Y)|^2 \delta(Y) dY
\lesssim
\iint_{B(x,r) \cap 2I_x} \delta(Y)^{-1} dY
\\
\approx
\ell(I_x)^{-1} \,r^{n} 
\approx
\delta(x)^{-1} \,r^{n} 
\lesssim
\gamma^\frac1n\,r^{n-1}.
\end{multline}
On the other hand, if $\gamma^\frac1n\delta(x)\le r$ we note that 
\[
B(x,r) \cap \Omega_* \subset 2I_x\cap \Omega_* 
\subset
\bigcup_{Q'\in\dd_{\F,Q}} (U_Q^{*}\cap 2I_x).
\]
It is clear that from construction if $U_{Q'}^{*}\cap 2I_x\neq\emptyset$ then $\ell(Q')\approx\ell(I_x)\approx \delta(x)$. Note also that $\#\{I\in\W: I\cap 2I_x\neq\emptyset\}\lesssim C_n$ hence $\#\{Q'\in\dd: U_{Q'}^{*}\cap 2I_x\neq\emptyset\}\lesssim C_{n,\eta,K}$. Thus, observing that $Q'\in\dd_{\F,Q'}$ for every $Q'\in \dd_{\F,Q}$ we obtain from \eqref{H:small} 
\begin{multline}\label{Case1b}
\iint_{B(x,r) \cap \Omega_*} |\nabla \wcalA(Y)|^2 \delta(Y) dY
\le
\sum_{\substack{Q'\in\dd_{\F,Q}\\ U_{Q'}^*\cap 2I_x\neq\varnothing}}
\iint_{U_Q^*} |\nabla \wcalA(Y)|^2 \delta(Y) dY
=
\sum_{\substack{Q'\in\dd_{\F,Q}\\ U_{Q'}^*\cap 2I_x\neq\varnothing}} \alpha_{Q'}
\\
\le
\sum_{\substack{Q'\in\dd_{\F,Q}\\ U_{Q'}^*\cap 2I_x\neq\varnothing}} m(\dd_{\F,Q'})
\le
\gamma\,\sum_{\substack{Q'\in\dd_{\F,Q}\\ U_{Q'}^*\cap 2I_x\neq\varnothing}}  \sigma(Q')
\lesssim
\gamma\,\,\delta(x)^{n-1}
\lesssim
\gamma^\frac1n\,r^{n-1}.
\end{multline}

\noindent\textbf{Case 2}. $\widetilde{M}^{-1}\,\ell(Q)<r<\diam(\pom_*)\approx\ell(Q)$. This is a trivial case since by construction and \eqref{H:small}  we obtain
\begin{multline}\label{Case2}
\iint_{B(x,r) \cap \Omega_*} |\nabla \wcalA(Y)|^2 \delta(Y) dY 
\le
\sum_{Q'\in\dd_{\F,Q}} \iint_{U_{Q'}^*} |\nabla \wcalA(Y)|^2 \delta(Y) dY 
\\
=
\sum_{Q'\in\dd_{\F,Q}} \alpha_{Q'}
=
m(\dd_{\F,Q})
\le
\gamma\sigma(Q)
\approx
\gamma \ell(Q)^{n-1}
\approx
\gamma r^{n-1}.
\end{multline}

\noindent\textbf{Case 3}. $c\delta(x)<r \le \widetilde{M}^{-1}\,\ell(Q)$. Pick $\hat x\in \pom$ such that $|x-\hat x| = \delta(x)$ and note that $B(x,r)\subset B(\hat{x}, (1+ c^{-1})r)$. Note also that if $Q'\in\dd_{\F,Q}$ is so that $U_{Q'}^*\cap B(x,r)\neq\emptyset$ then we can find $I\in \W_{Q'}$ and $Y\in I^{**}\cap B(x,r)$ so that by \eqref{def:WQ}
\[
\eta^{\frac12}\ell(Q')
\lesssim 
\ell(I)
\approx
\delta(Y)
\le
|Y-x|+\delta(x)
<
(1+c^{-1})\,r.
\]
and for every $y\in Q'$
\begin{multline*}
|y-\hat{x}|
\le
\diam(Q')+\dist(Q',I)+\diam(I)+|Y-x|+|x-\hat{x}|
\\
\lesssim
K^{\frac12}\ell(Q')+r+\delta(x)
\lesssim
K^{\frac12}\,\eta^{-\frac12}(1+ c^{-1})\,r.
\end{multline*}
Consequently, if we write $\widetilde{M}'= C\,K^{\frac12}\,\eta^{-\frac12}\,(1+c^{-1})$ and choose $\widetilde{M}>\widetilde{M}'$, it follows that $\ell(Q')<\widetilde{M}'\,r<\ell(Q)$
and $Q'\subset \Delta(\hat{x}, \widetilde{M}'r)=:\Delta'$. 

We can then find a pairwise disjoint family of dyadic cubes $\{Q_k\}_{k=1}^{\widetilde{N}}$ with uniform cardinality $\widetilde{N}$ (depending on $C_{AR}$ and $n$) so that $2^{-1}\,\widetilde{M}'\,r\le  \ell(Q_k)< \widetilde{M}'\,r $, $Q_k\cap \Delta'\neq\emptyset$ for every $1\le k\le \widetilde{N} $, and $\Delta'\subset \bigcup_{k=1}^{\widetilde{N}} Q_k$. 
Relabeling if necessary, we can assume that there exists $\widetilde{N}'\le \widetilde{N}$ so that $\Delta'\cap Q\subset \bigcup_{k=1}^{\widetilde{N}'} Q_k$ and each $Q_k$ meets $\Delta'\cap Q$ for $1\le k\le \widetilde{N}'$. We would like to observe that necessarily $\widetilde{N}'\ge 1$ since we have shown that $Q'\subset \Delta'$ for every $Q'\in\dd_{\F,Q}$ so that $U_{Q'}^*\cap B(x,r)\neq\emptyset$. Also $Q_k\subset Q$ for $1\le k\le \widetilde{N}'$ since $\ell(Q_k)<\widetilde{M}'\,r<\ell(Q)$ and $Q_k$ meets $Q$.
Moreover, for every such a $Q'$ we necessarily have $Q'\subset Q_k$ for some $1\le k\le \widetilde{N}'$ since
$Q'\subset \Delta'\cap Q$, hence $Q'$ meets some $Q_k$ and also $\ell(Q')<\widetilde{M}'\,r\le 2\ell(Q_k)$ which forces $Q'\subset Q_k$. All these and \eqref{H:small} readily imply
\begin{multline}\label{Case3}
\iint_{B(x,r) \cap \Omega_*} |\nabla \wcalA(Y)|^2 \delta(Y) dY 
\le
\sum_{\substack{Q'\in\dd_{\F,Q}\\ U_{Q'}^*\cap B(x,r)\neq\varnothing}} \iint_{U_{Q'}^*} |\nabla \wcalA(Y)|^2 \delta(Y) dY 
=
\sum_{\substack{Q'\in\dd_{\F,Q}\\ U_{Q'}^*\cap B(x,r)\neq\varnothing}} \alpha_{Q'}
\\
\le
\sum_{k=1}^{\widetilde{N}'}
\sum_{\substack{Q'\in\dd_{\F,Q}\\ Q'\subset Q_k\subset Q}} \alpha_{Q'}
=
\sum_{k=1}^{\widetilde{N}'}
\sum_{Q'\in\dd_{\F,Q_k}} \alpha_{Q'}
=
\sum_{k=1}^{\widetilde{N}'} m(\dd_{\F,Q_k})
\le
\gamma \sum_{k=1}^{\widetilde{N}'} \sigma(Q_k)
\lesssim
\gamma r^{n-1}.
\end{multline}

\medskip

Combining what we have obtained in all the cases we see that \eqref{Case1a}, \eqref{Case1b}, \eqref{Case2}, and \eqref{Case3} give, since $0<\gamma<1$, that 
\[
\frac1{r^{n-1}}\iint_{B(x,r) \cap \Omega_*} |\nabla \wcalA(Y)|^2 \delta(Y) dY
\le
C_0\gamma^{\frac1n},
\]
for some constant $C_0\ge 1$ depending on the allowable constants and where we recall that $\gamma$ is at our choice. Hence we just need to pick $\gamma< (C_0^{-1}\epsilon)^n$ to conclude as desired \eqref{sC}.

\section{Transference of the \texorpdfstring{$A_\infty$}{A-infinity} property to sawtooth domains}\label{sect:trsf}

In this section we show that the $A_\infty$ property for the elliptic operator $L$ in $\Omega$ can be transferred to sawtooth subdomains with constants that only depend on the allowable constants. We first work with sawtooth subdomains which are compactly contained in $\Omega$ and then we consider the general case using that interior sawtooth subdomains exhaust general sawtooth domains.

\begin{theorem}\label{thm:trsf}
Let $\Omega\subset \RR^n$ be a uniform domain with Ahlfors regular boundary. Let $\wcalA$ be a \textbf{symmetric} uniformly elliptic matrix on $\Omega$ and $L=-\divg(\wcalA\nabla)$ . Assume the following two properties: 
\begin{enumerate}[label=\textup{(\arabic*)}, itemsep=0.2cm] 
		\item\label{1-thm:trsf} The elliptic measure $\omega_L$ associated with the operator $L$ relative to the domain $\Omega$ is of class $A_\infty$ with respect to the surface measure. 
		
		\item\label{2-thm:trsf} For every fundamental chord-arc subdomain $\PP$ of $\Omega$, see \eqref{fundamental-casd}, the elliptic measure associated with $L$ relative to the domain $\PP$ is also of class $A_\infty$ with respect to the surface measure of $\PP$, with uniform $A_\infty$ constants.
\end{enumerate}
	For every $Q\in \DD$ and every family of pairwise disjoint dyadic subcubes $\F = \{Q_j\} \subset \DD_Q$, let $\Omega_*=\Omega_{\F,Q}$ (or $\Omega_*=\Omega_{\F,Q}^*$) be the associated sawtooth domain, and $\omega_*$ and $\sigma_* = \HH^{n-1}|_{\partial\Omega_*}$ be the elliptic measure for $L$ and the surface measure of $\Omega_*$. Then $\omega_* \in A_\infty(\sigma_*)$, with the $A_\infty$ constants independent of $Q$ and $\F$.
\end{theorem}

Note that if $\wcalA$ is a non necessarily symmetric matrix satisfying hypotheses \ref{H1} and \ref{H2} in $\Omega$, we can easily verify it also satisfies the Kenig-Pipher condition relative to every fundamental chord-arc subdomain. Indeed, since $\PP\subset \Omega$ then $\delta_{\PP}(\cdot)\le \delta(\cdot)$ and \ref{H1} in $\PP$ is automatic. On the other hand, let $\PP=\interior \left( \bigcup_{j=1}^{m_1} I_j^\ast\right)$  with  $I_j\in\W $ and $ I\cap I_j \not =\emptyset$ and take $x\in\partial\PP$ and $r\le \diam\PP\lesssim \ell(I)$. Note that  \ref{H1} implies  $|\nabla \wcalA(Y)|^2\lesssim \ell(I)^{-2} $ for every $Y\in\PP$ since $\delta(Y)\approx\ell(I)$ , hence 
\begin{equation}\label{carleson-in-p}
\frac{1}{r^{n-1}}\iint_{B(x,r)\cap\PP}|\nabla \wcalA(Y)|^2 \delta_{\PP}(Y) dY\lesssim\frac{1}{r^{n-1}\ell(I)^2}\iint_{B(x,r)\cap\PP} \delta_{\PP}(Y) dY\lesssim 
\frac{r^{n+1}}{r^{n-1}\ell(I)^2}\lesssim 1.
\end{equation}
That is, \ref{H1} in $\PP$ holds as well. Thus by \cite{KP} (and the slight improvement in \cite{HMT1}), and the fact that chord-arc domains can be approximated by Lipschitz domains, one obtains that the elliptic measure for $L$ relative to $\PP$ is also of class $A_\infty$ with respect to the surface measure of $\PP$ and \ref{2-thm:trsf} in the previous result holds.  On the other hand, \cite[Theorem 1.6]{CHMT} asserts that for any uniform domain $\Omega$, and under the assumptions   \ref{H1} and \ref{H2}, one has that $\omega_L\in A_\infty(\sigma)$ if and only if $\omega_{L^{\rm sym}}\in A_\infty(\sigma)$ where $L^{\rm sym}$ is the operator associated with the symmetric matrix $\wcalA^{\rm sym}=(\frac{a_{ij}+a_{ji}}2)_{i,j=1}^n$. Note that $\wcalA^{\rm sym}$ is also a uniformly elliptic matrix in $\Omega$ with the same ellipticity constants as $\wcalA$ and satisfies  \ref{H1} and \ref{H2} with constants which are controlled by those of $\wcalA$. With all these observations we immediately get the following corollary:  
\begin{corollary}\label{corol:KP-transfer}
	Let $\Omega\subset \RR^n$ be a uniform domain with Ahlfors regular boundary. Suppose that $\wcalA$ is a (non necessarily symmetric) uniformly elliptic matrix on $\Omega$ 
	satisfying the hypotheses \textup{\ref{H1}} and \textup{\ref{H2}}, and that the elliptic measure $\omega_L$ associated with the operator $L$ relative to the domain $\Omega$ is of class $A_\infty$ with respect to the surface measure. Then the elliptic measure associated with $L$ relative to any sawtooth domain is of class $A_\infty$, with uniform constants.
\end{corollary}

\begin{proof}[Proof of Theorem \ref{thm:trsf}]

The proof Theorem \ref{thm:trsf} has several steps. We work with $\Omega_*=\Omega_{\F, Q}$ as the proof with $\Omega_{\F, Q}^*$ is identical. We first assume that the sawtooth domain $\Omega_\ast=\Omega_{\F, Q}$ is \textbf{compactly contained} in $\Omega$ and show that $\omega_\ast \in A_\infty(\sigma_*)$, with the $A_\infty$ constants independent of $Q$ and $\F$. Here we  use $\omega_*$ to denote the elliptic measure associated with $L$ relative to $\Omega_*$. % Then in the case when $\Omega_\ast$ is not compactly contained in $\Omega$, we approximate $\Omega_\ast$ by interior compactly contained sawtooth domains in $\Omega$. We apply the result from the first case and use Theorem \ref{thm:CMEAinfty} to show that the $A_\infty$ constants are preserved when passing to the limiting domain.

Under the assumption that $\Omega_\ast$ is compactly contained in $\Omega$, for $Q$ fixed, let $N$ be an integer such that $\dist(\Omega_*,\pom) \approx 2^{-N} \ell(Q)$. 
Then $\Omega_*$ if formed by a union of fattened Whitney boxes of side length  controlled from below by $c\,2^{-N} \ell(Q)$ hence $\Omega_*$ clearly 
satisfies a \textit{qualitative exterior corkscrew condition}, that is, it satisfies the exterior corkscrew condition for surface balls up to a scale of the order of $2^{-N}\ell(Q)$. In the case of  the Kenig-Pipher operators, this information alone does not suffice to derive the desired $A_\infty$ property, with constant independent of $N$; however this does give us the qualitative absolute continuity $\omega_*^X \ll \sigma_*$ for any $X\in \Omega_*$ (since $\Omega_*$ is a chord-arc domain with constants depending on $N$). Note that Theorem \ref{thm:trsf} is nonetheless written for a more general class and it is not obvious whether we can automatically have the desired absolute continuity. This will be shown in the course of the proof.

Our main task is to then show that $\omega_* \in A_\infty(\sigma_*)$ with constants that depend only on the allowable constants. If we write $\textbf{k}_*:=d\omega_*/d\sigma_*$ for the Radon-Nikodym derivative, by the change of pole  formula Lemma \ref{lm:cop}, obtaining 
$\omega_* \in A_\infty(\sigma_*)$, it is equivalent to prove the following: 
there exists an exponent $p\in (1,\infty)$ and a constant $C$ depending only on the allowable constants such that
for any surface ball 
$\Delta_*=B_*\cap\partial\Omega_*$ centered at $\pom_*$, with radius smaller than the diameter of $\pom_*$, and for $X=X_{\Delta_*}\in \Omega_*\cap B_*$, a corkscrew point relative to $\Delta_*$, the following holds
\begin{equation}\label{RH}
	\int_{\Delta_*} \left(\textbf{k}_*^X\right)^p d\sigma_* \leq C \sigma_*(\Delta_*)^{1-p}.
\end{equation}
Since $\diam(\Omega_*) \approx \ell(Q)$, it is easy to see by a standard covering argument and Harnack's inequality that it suffices to prove \eqref{RH} for 
$r_* \le M_1^{-1} \ell(Q)$, 
where $M_1$ is a suitably large fixed constant.  By hypothesis \ref{1-thm:trsf}, $\omega_L\in A_\infty(\sigma)$, hence it belongs to the reverse H\"older class with some exponent $p_1>1$ (see \eqref{eq1.wRH}). Also, by hypothesis \ref{2-thm:trsf} we know that the elliptic measure relative to any fundamental chord-arc subdomain $\PP$ satisfies an $A_\infty$ condition with respect to the corresponding surface measure with uniform bounds. In turn, there exists $p_2>1$ and a uniform constant so that any of these elliptic measures belong to the reverse Hölder class with this exponent $p_2$ and with the same uniform constant  (see \eqref{eq1.wRH}).  We shall henceforth set $p:= \min \{p_1,p_2\}$, and it 
is for this $p$ that we shall prove \eqref{RH}. 

To start with the proof, recall that as observed above, since $\dist(\Omega_*, \pom) \approx 2^{-N}\ell(Q)$, it follows that all the dyadic cubes $Q'\in \DD_{\F,Q}$ have length $\ell(Q') \gtrsim 2^{-N}\ell(Q)$, and the cardinality of $\DD_{\F,Q}$ is bounded by a constant $C(N)$. Hence $\Omega_*=\Omega_{\F,Q}$ is formed by the finite union of Whitney regions $U_{Q'}$ with $Q' \in \DD_{\F,Q}$ satisfying $\ell(Q') \gtrsim 2^{-N}\ell(Q)$. In turn each $U_{Q'}$ is a polyhedral domain consisting of a finite number of fattened Whitney boxes with side length of the order of $\ell(Q')$. In particular there exists a finite index set $\N_*$ so that 
\begin{equation}\label{bdry-Omega*}
	\pO_* \subset \bigcup_{i\in \N_*} \partial \left((I^i)^* \right) \cap \pO_* = : \bigcup_{i\in \N_*} S_*^i.
\end{equation}
where $S_*^i \neq \emptyset$ for each $i\in \N_*$, $\interior((I^i)^*)\subset \Omega_*$, and $\ell(I^i)\gtrsim 2^{-N}$. 
For each $I^i$, with $i\in \N_*$, we pick $Q^i\in\DD_{\F,Q}$ such that $\W_{Q^i} \ni I^i$ (there could be more than one such a $Q^i$ in which case we just select one).  Note that different $I^i$'s may correspond to the same $Q^i$, but each $Q^i$ may only repeat up to a finitely many times, depending only on the allowable constants.  
Since $S_*^i$ is contained in the boundary of a fattened Whitney box $(I^i)^*$,
\begin{equation}\label{Sitmp}
	\diam(S_*^i) \lesssim \ell(I^i) \approx \ell(Q^i), \text{ and } \dist(S_*^i, \pO) \geq \dist((I^i)^*, \pO) \approx \ell(Q^i). 
\end{equation} 
On the other hand, the fact that $S_*^i \subset \pO_*$ means that $I^i$ intersects some $J^i\in \W$ so that if $J^i\in \W_{Q''}$ then  $Q'' \notin \DD_{\F,Q}$.
If we pick $\widetilde{Q}^i\in\dd$ so that $\ell(\widetilde{Q}^i)=\ell(J^i)$ and $\dist(J^i,\pom)=\dist(J^i, \widetilde{Q}^i)$ then as mentioned right below \eqref{def:WQ0} we have that $J^i\in \W^0_{\widetilde{Q}^i}\subset \W_{\widetilde{Q}^i}$, therefore $\widetilde{Q}^i\notin \DD_{\F,Q}$. Recalling \eqref{eq:tausmall2} and the comments after it, we know that $tJ^i \subset \Omega\setminus\Omega_*$ and  $\partial \left( (I^i)^* \right) \cap \pO_*$ contains  
an $(n-1)$-dimensional ball  with radius of the order of $\min\{\ell(I^i), \ell(J^i)\}\approx \ell(I^i)$. Denote that $(n-1)$-dimensional ball by  $\Delta_*^i \subset S_*^i$. 
This implies, combined with \eqref{Sitmp}, that
\begin{equation}\label{eq:sizeSi}
	r(\Delta_*^i) \approx \diam(S_*^i) \approx \ell(I^i) \approx \ell(Q^i)\ \text{\ and\ } \dist(S_*^i, \pO) \approx \dist(\Delta_*^i, \pO) \approx \ell(I^i) \approx \ell(Q^i).
\end{equation}

At this stage we consider several cases. In the \textbf{Base case}, see Lemma \ref{lemma:base-case}, we treat surface balls $\Delta_*$ with small radii so that $\Delta_*$ is contained in a uniformly bounded union of Whitney cubes of comparable sides. In the case when $\Delta_*$ is large we decompose the intersection of $\Delta_*$ in small pieces to which the base case can be applied (\textbf{Step 1}). We then put all the local estimates together to obtain a global one (\textbf{Step 2}). This requires Lemma \ref{lm:thinbd} and to consider several cases to account for all the small pieces.

Let $\Delta_* = B^*\cap  \pO_* \subset \Omega$, with $B_*=B(x_*,r_*)$, $x_*\in\pO_*$ and $0<r_*<\diam(\pom_*)$. Since $\Omega_*$ is a uniform domain (see Lemma \ref{lemma:sawtooth-inherit}), we can pick $X_{\Delta_*}\subset B_*\cap\Omega_*$, a Corkscrew point relative to $\Delta_*$ in $\Omega_*$, so that
$\delta_*(X):=\dist(X,\pom_*) \approx r_*$. 
 Write
\begin{equation}\label{eq:cover}
\Delta_* \subset \bigcup_{i\in \N_{\Delta_*}} S_*^i, \quad \text{ where }  \N_{\Delta_*}:=\{i:\in \N_{*}: \Delta_* \cap S_*^i \neq \emptyset\}.
\end{equation}

\begin{lemma}[Base case]\label{lemma:base-case}
Using the notation above we have that $\omega_*\ll \sigma_*$ in $\pom_*$.
Moreover if there exists  $i\in \N_{\Delta_*}$ such that $r_*\leq \frac{\tau}{8} \ell(I^i)$ then
\begin{equation}\label{conc-base-case}
\int_{\Delta_*} \left( \textup{\textbf{k}}_*^{X_{\Delta_*}} \right)^p d\sigma_*
\lesssim
\sigma_*(\Delta_*)^{1-p}\,,
\end{equation}
where $\textup{\textbf{k}}_* := d\hm_*/d\sigma_*$, $p$ is as above, and the implicit constant only depends on the allowable constants.
\end{lemma}

\begin{proof}
We first claim that 
\begin{equation}\label{eq:coversmall}
	2\Delta_* \subset \bigcup_{\substack{i'\in \N_*\\ I^{i'} \cap I^i \neq \varnothing} } S_*^{i'} \quad {\rm and}
	\quad 2B_*\cap\Omega_* \subset \bigcup_{\substack{i'\in \N_* \\ I^{i'} \cap I^i \neq \varnothing} } (I^{i'})^*.
\end{equation}
In fact, for any $i' \in \N_*$, if $S_*^{i'}$ intersects $2\Delta_*$, or if
$2B_*\cap\Omega_*$ intersects $ (I^{i'})^*$, then our current assumption gives 
\[
	\dist\big( (I^i)^*, (I^{i'})^* \big) \leq \dist \big((I^i)^* \cap 2B_*, (I^{i'})^* \cap 2B_* \big)
	 \leq \diam(2B_*) = 4r_*\leq \frac{\tau}{2} \ell(I^i),
\]
and thus 
\[ (I^i)^{**} \cap (I^{i'})^{**}\,\supset\, (I^i)^{**} \cap (I^{i'})^*\, \neq\, \emptyset. \]
By the choice of $\tau$, i.e., by \eqref{eq:tausmall1},  we then have $I^i \cap I^{i'} \neq \emptyset$ and the claim is proved. 
Next, let $m_1$ denote the maximal number of Whitney boxes intersecting $I^i$. Note that $m_1$ only depends on the constructions of the Whitney cubes, hence just on dimension. By relabeling \eqref{eq:coversmall} we write
\begin{equation}\label{eq:coversmall2}
	2\Delta_* \subset \bigcup_{i'=1}^{m_1} S_*^{i'}\quad {\rm and} \quad 2B_*\cap\Omega_* 
	\subset \bigcup_{i'=1}^{m_1}(I^{i'})^*.
\end{equation}
Moreover by \eqref{eq:sizeSi}, for each $i'=1, \dots, m_1$, we have $
	\diam(S_*^{i'}) \approx \ell(I^{i'}) \approx \ell(I^i)$.
Set  then 
$$\PP:= \interior\left( \bigcup_{i'=1}^{m_1} (I^{i'})^*\right)\subset\Omega_*,$$
which by construction is a fundamental chord-arc subdomain $\PP$ of $\Omega$, see \eqref{fundamental-casd}. 
Note that since $2B_*\cap\Omega_* $ is open then \eqref{eq:coversmall2} says that $2B_*\cap\Omega_* \subset \PP$ and hence $2B_*\cap\Omega_* =2B_*\cap\PP$. This and the fact that $\Omega_*$ and $\PP$ are open readily implies that $2B_*\cap\partial\Omega_* =2B_*\cap\partial\PP$. Moreover, $X_{\Delta_*}\in\PP$ (since $X\in B_*\cap\Omega_*$)
and 
\begin{equation}\label{eq5.13}
\dist(X_{\Delta_*},\partial\PP) \approx \delta_*(X_{\Delta_*}) \approx r_*\leq \frac{\tau}{8} \ell(I^i) \leq
\frac{\tau}{8}\diam(\PP)
\le 
\,\diam(\PP).
\end{equation}
Let $X_{\PP}$ be a Corkscrew point for the domain $\PP$, at the scale $ \ell(I^i) \approx\diam(\PP)$, i.e.,
$X_\PP$ is a Corkscrew point in $\PP$ relative to the surface ball consisting of the entire
boundary of $\PP$.  Thus in particular, $\dist(X_\PP,\partial\PP)\approx \diam(\PP)\ge \frac8{\tau} r_*$, hence $\dist(X_\PP,\partial\PP)\ge 2\,c_0\,r_*$ for some uniform $0<c_0<1/4$.  

Set $u_1(\cdot):=G_*(X_{\PP}, \cdot)$ and $u_2(\cdot):=G_{\PP}(X_{\PP},\cdot)$ in $2B_*\cap\Omega_* =2B_*\cap\PP$ where $G_*$ and $G_\PP$ are the Green functions for the operator $L$ and for the domains $\Omega_*$ and $\PP$ respectively, and where as observed above $X_\PP\in\PP\subset\Omega_*$. Fix $y\in\frac32B_*\cap\partial\Omega_*=\frac32B_*\cap\partial\PP$ and  note that $B(y,c_0 r_*)\subset 2B_*$. Note that if $Z\in B(y,c_0 r_*)$ then 
\[
2\,c_0r_*\le \dist(X_\PP,\partial\PP)
\le
|X_\PP-y|
\le
|X_\PP-Z|+|Z-y|
<
|X_\PP-Z|+c_0 r_*,
\]
and $|X_\PP-Z|> c_0 r_*$. As a consequence, $B(y,c_0 r_*)\subset 2B_*\setminus B(X_{\PP}, c_0 r_*)$. Hence, $L u_1=0$ and $L u_2=0$ in we weak sense in $B(y,c_0 r_*)\cap\Omega_* =B(y,c_0 r_*)\cap\PP$ and both are continuous in $B(y,c_0 r_*)\cap\overline{\Omega_*} =B(y,c_0 r_*)\cap\overline{\PP}$. In particular both vanish continuously in 
$B(y,c_0 r_*)\cap\partial\Omega_* =B(y,c_0 r_*)\cap\partial\PP$. This means that we can  use Lemma \ref{lm:comp} in $\mathcal{D}=\PP$ to obtain that for every $Z\in B(y, c_0\,r_*/8)$
\begin{equation}\label{comp-u1-u2}
\frac{u_1(Z)}{u_2(Z)}
\approx
\frac{u_1(X_{\Delta_\PP(y,c_0 r_*/2)}^\PP)}{u_2(X_{\Delta_\PP(y,c_0 r_*/2)}^\PP)}
\end{equation}
where $X_{\Delta_\PP(y,c_0 r_*/2)}^\PP$ is a corkscrew relative to $B(y,c_0 r_*/2)\cap\overline{\PP}$ for the fundamental chord-arc domain $\PP$. On the other hand Lemma \ref{lm:CFMS} applied in $\Omega_*$ (which is uniform with Ahlfors regular boundary and the implicit constants are uniformly controlled, see Lemma \ref{lemma:sawtooth-inherit}) and $\PP$ (a fundamental chord-arc domain) gives for any $0<s\le c_0 r_*/2$
\begin{equation}\label{CFMS-u1-u2}
u_1(X_{\Delta_*(y,s)}^*)\approx \omega_*^{X_{\PP}}(\Delta_*(y,s))\,s^{n-2},
\qquad 
u_2(X_{\Delta_\PP(y,s)}^\PP)\approx \omega_\PP^{X_{\PP}}(\Delta_\PP(y,s))\,s^{n-2},
\end{equation}
where $X_{\Delta_*(y,s)}^*$ is the corkscrew point relative to  $\Delta_*(y,s)= B(y,s)\cap\pom_*$ for the uniform domain $\Omega_*$,
$X_{\Delta_\PP(y,s)}^\PP$ is the corkscrew point relative to  $\Delta_\PP(y,s)= B(y,s)\cap\partial\PP$ for the fundamental chord-arc uniform domain $\PP$, and $\omega_\PP$ stands for the elliptic measure associated with the operator $L$ relative to $\PP$. Note that from the definition of corkscrew condition and the fact that $B(y,s)\cap\Omega_*=B(y,s)\cap\PP$ it follows that $X_{\Delta_*(y,s)}^*, X_{\Delta_\PP(y,s)}^*\in B(y,s)\cap\Omega_*=B(y,s)\cap\PP$ and also 
\[
\dist(X_{\Delta_*(y,s)}^*, \pom_*)\approx \dist(X_{\Delta_*(y,s)}^*, \partial\PP)\approx \dist(X_{\Delta_\PP(y,s)}^\PP, \pom_*)\approx
\dist(X_{\Delta_\PP(y,s)}^\PP, \partial\PP)\approx s.
\]
Consequently $u_1(X_{\Delta_*(y,s)}^*)\approx u_1(X_{\Delta_\PP(y,s)}^\PP)$ and $u_1(X_{\Delta_\PP(y,c_0 r_*/2)}^\PP)\approx u_1(X_{\Delta_*(y,c_0 r_*/2)}^*)$.
All these, together with \eqref{comp-u1-u2}, \eqref{CFMS-u1-u2}, and Lemma \ref{lm:Bourgain}, give for every  $0<s\le c_0 r_*/8$
\begin{multline}\label{comp-w*-wPP}
\frac{\omega_*^{X_{\PP}}(\Delta_*(y,s))}{\omega_\PP^{X_{\PP}}(\Delta_\PP(y,s))}
\approx
\frac{u_1(X_{\Delta_*(y,s)}^*)}{u_2(X_{\Delta_\PP(y,s)}^\PP)}
\approx
\frac{u_1(X_{\Delta_\PP(y,s)}^\PP)}{u_2(X_{\Delta_\PP(y,s)}^\PP)}
\approx
\frac{u_1(X_{\Delta_\PP(y,c_0 r_*/2)}^\PP)}{u_2(X_{\Delta_\PP(y,c_0 r_*/2)}^\PP))}
\\
\approx
\frac{u_1(X_{\Delta_*(y,c_0 r_*/2)}^*)}{u_2(X_{\Delta_\PP(y,c_0 r_*/2)}^\PP))}
\approx
\frac{\omega_*^{X_{\Delta_*(y,c_0 r_*/2)}}(\Delta_*(y,c_0 r_*/2))}{\omega_\PP^{X_{\Delta_\PP(y,c_0 r_*/2)}}(\Delta_\PP(y,c_0r_*/2))}
\approx 1
.
\end{multline}
With this in hand, we note that since $y\in\Delta_*(x_*,\frac32r_*)=\frac32B_*\cap\partial\Omega_*=\frac32B_*\cap\partial\PP=\ \Delta_\PP(x_*,\frac32r_*)$ and $0<s\le c_0 r_*/8$ are arbitrary we can easily conclude, using a Vitali covering argument and the fact that both $\omega_*^{X_{\PP}}$ and $\omega_\PP^{X_{\PP}}$ are outer regular  and doubling in $\Delta_*(x_*,\frac32r_*)=\Delta_\PP(x_*,\frac32r_*)$, that $\omega_*^{X_{\PP}}(F)\approx \omega_\PP^{X_{\PP}}(F)$ for any Borel set $F\subset \Delta_*(x_*,\frac32r_*)=\Delta_\PP(x_*,\frac32r_*)$. Hence $\omega_*^{X_{\PP}}\ll \omega_\PP^{X_{\PP}} \ll\omega_*^{X_{\PP}}$ in $\Delta_*(x_*,\frac32r_*)=\Delta_\PP(x_*,\frac32r_*)$. From hypothesis  \ref{2-thm:trsf} in Theorem \ref{thm:trsf} we know that  $\omega_\PP\ll \sigma_\PP:=\HH^{n-1}|_{\partial\PP}$, hence in particular  $\omega_*\ll\sigma_*$ in $\Delta_*(x_*,\frac32r_*)$. This, \eqref{comp-w*-wPP}, and Lebesgue's differentiation theorem readily imply that 
\begin{equation}\label{tmp1}
	\textbf{k}_*^{X_{\PP}} (y) \approx \textbf{k}_{\PP}^{X_{\PP}}(y), \quad \text{for $\HH^{n-1}$-almost all } y\in \Delta_*(x_*,r_*)=\Delta_\PP(x_*,r_*),
\end{equation} 
where $\textbf{k}_\PP := d\hm_\PP/d\sigma_\PP$ and $\textbf{k}_* := d\hm_*/d\sigma_\PP$. 

We next observe that Lemma \ref{lm:cop} applied with $\mathcal{D}=\Omega_*$ (along with Harnack's inequality  for the case $r_*\approx \ell(I^i)$) and Lebesgue's differentiation theorem  yield 
	\begin{equation}\label{tmp2}
		\textbf{k}_*^{X_{\Delta_\ast}}(y) \approx \frac1{\omega_*^{X_\PP}(\Delta_*)} \, \textbf{k}_*^{X_{\PP}}(y)\quad \text{for $\sigma_*$-almost all } y\in \Delta_*. 
	\end{equation} 
Since $\PP$ is a fundamental chord-arc subdomain $\PP$ of $\Omega$, see \eqref{fundamental-casd}, as observed above $\omega_\PP$ belongs to the reverse Hölder class 
with exponent $p_2>1$ and so with exponent $p= \min \{p_1, p_2\}$. We find that since ${\sigma_\ast} ={\sigma_\PP}$ in $\Delta_\ast=\Delta_*(x_*,r_*)=\Delta_{\PP}(x_*,r_*)$
\begin{multline*}%\label{tt-1}
		\int_{\Delta_*} \left( \textbf{k}_*^{X_{\Delta_\ast}} \right)^p d\sigma_*
		\lesssim
		\frac{1}{\big(\omega_*^{X_\PP}(\Delta_*)\big)^{p}} \,
		\int_{\Delta_*} \left( \textbf{k}_*^{X_\PP} \right)^p d\sigma_*
		\approx
		\frac{\sigma_\PP(\Delta_\PP(x_*,r_*))}{\big(\omega_*^{X_\PP}(\Delta_*)\big)^{p}} \,
		\fint_{\Delta_\PP(x_*,r_*)} \left( \textbf{k}_\PP^{X_\PP} \right)^p d\sigma_\PP\\
\lesssim 
\frac{\sigma_\PP(\Delta_\PP(x_*,r_*))}{\big(\omega_*^{X_\PP}(\Delta_*)\big)^{p}} 
\left(\frac{\omega_\PP^{X_\PP}(\Delta_\PP(x_*,r_*))}{\sigma_\PP(\Delta_\PP(x_*,r_*))}\right)^p
		\approx\, \sigma_*(\Delta_*)^{1-p}\,,
	\end{multline*}
where we have used \eqref{tmp2}, \eqref{tmp1}, that ${\sigma_\ast} ={\sigma_\PP}$,  the reverse Hölder estimate with exponent $p$ for $\textbf{k}_\PP$, and that both $\pom_*$ and $\partial\PP$ are Ahlfors regular sets with uniform bounds.

To complete our proof we need to see that $\omega_*\ll\sigma_*$ in $\pom_*$. Let us observe that we have already obtained that $\omega_*\ll\sigma_*$ in $\Delta_*(x_*,\frac32r_*)$ where $x_*\in \pom_*$ is arbitrary and $r_*\le\frac{\tau}{8}\ell(I^i)$ for some $i\in \mathcal{N}_{\Delta_*}$. We may cover $\pom_*$ by a finite union of surface balls $\Delta_*(x_j, r_j)$, with $r_j=\frac{2^{-N}}{\tilde{M}}\ell(Q)$, where $\widetilde{M}$ is large enough to be chosen, whose cardinality may depend on $N$ and $\widetilde{M}$. Note that for every $i\in \mathcal{N}_*$ we have, as observed before, that $\ell(I^i)\gtrsim 2^{-N}\ell(Q)> \frac{8}{\tau} \frac{2^{-N}}{\tilde{M}}\ell(Q)$ if we pick $\widetilde{M}$ large enough. Hence, for every $j$, it follows that $r_j<\frac{\tau}{8}\ell(I^i)$ for every $i\in\mathcal{N}_*$ and in particular for every $i\in \mathcal{N}_{\Delta_*(x_j,r_j)}$. Hence the previous argument yields that
$\omega_*\ll\sigma_*$ in $\Delta_*(x_j,\frac32r_j)$ for every $j$ and consequently $\omega_*\ll\sigma_*$ in $\pom_*$. 
\end{proof}

\begin{remark}
	We would like to emphasize that the fact that  $\omega_*\ll\sigma_*$ in $\pom_*$ is automatic for the Kenig-Pipher operators. In fact as observed above $\Omega_*$ is a chord-arc domain and hence $\omega_*\in A_\infty(\sigma_*)$ (albeit with constants which may depend on $N$). The previous argument proves that the more general hypothesis  \ref{2-thm:trsf} in Theorem \ref{thm:trsf} also yields $\omega_*\ll\sigma_*$ in $\pom_*$.
	 \end{remark}

Once the \textbf{Base case} has been established we can focus on proving the $A_\infty$ property for the sawtooth. With this goal in mind we fix a surface ball $\Delta_* = B_*\cap  \pO_* \subset \Omega$, with $B_*=B(x_*,r_*)$, $x_*\in\pO_*$ and $0<r_*<\diam(\pom_*)$. Let $X:=X_{\Delta_*}\subset B_*\cap\Omega_*$ be a Corkscrew point relative to $\Delta_*$ in $\Omega_*$, so that
$\delta_*(X) \approx r_*$. Our goal is to show \eqref{RH}. As explained above we may assume that $r_* \le M_1^{-1} \ell(Q)$, for some $M_1$ large enough to be chosen. The \textbf{Base case} (Lemma \ref{lemma:base-case})
yields \eqref{RH} when $r_* <\frac{\tau}{8} \ell(I^i)$ for some $i\in \N_{\Delta_*}$. Hence we may assume from now on that $r_* \geq \frac{\tau}{8} \ell(I^i)$ for every $i\in \N_{\Delta_*}$.

\noindent\textbf{Step 1.} Show that 
\begin{align}\label{eq:tosumN}
\int_{\Delta_*} \left(\textbf{k}_*^X \right)^p d\sigma_* \lesssim \sum_{i\in \N_{\Delta_*}} \int_{Q^i} \left(\textbf{k}^X \right)^p d\sigma,
\end{align}
where we recall that $Q^i\in\dd_{\F,Q}$ is so that $I^i\in\W_{Q^i}$ for every $i\in \mathcal{N}_{*}$, and where $\textbf{k}=d\omega_L/d\sigma$.

To see this, by \eqref{eq:cover}, it suffices to obtain 
	\begin{equation}\label{eq:Si} 
\int_{S_*^i} \left( \textbf{k}_*^X \right)^p d\sigma_* 
\lesssim \int_{Q^i} \left(\textbf{k}^X \right)^p d\sigma.
\end{equation}
for each $i\in \N_{\Delta_*}$. Fix then such an $i$ and cover $S_*^i$ by a uniformly bounded number of
surface balls centered at $\pom_*$ with small radius $\Delta_*^{i,l}=B_*^{i,l}\cap \pO_*$ where $S_*^i \cap \Delta_*^{i,l} \neq \emptyset$ and $r(\Delta_*^{i,l}) \approx c \diam(S_*^i)\approx c\,\ell(I^i) $, the constant $c$ is chosen sufficiently small 
(depending on $\tau$), so that $r(\Delta_*^{i,l})\ll (\tau/8)\ell(I^i) $. Hence in the present scenario,
\begin{equation}\label{eq5.20}
\delta_*(X)\approx r_* \gg r(\Delta_*^{i,l})\,.
\end{equation}
We further choose $c$ small enough so that
\begin{equation}\label{eq5.21}
		2\Delta_*^{i,l} \subset \bigcup_{\substack{i'\in \N_* \\ I^{i'} \cap I^i\neq \varnothing} } S_*^{i'}\quad 
 {\rm and}
	\quad 2B^{i,l}_*\cap\Omega_* \subset 
	\bigcup_{\substack{i'\in \N_* \\ I^{i'} \cap I^i \neq \varnothing} } (I^{i'})^*.
	\end{equation} 
Note that there are at most a uniformly bounded number of such $i'$, for each $l$.  In each $\Delta_*^{i,l}$ we can use the \textbf{Base Case}, Lemma \ref{lemma:base-case}, since by construction $r(\Delta_*^{i,l})\ll (\tau/8)\ell(I^i) $ and hence \eqref{conc-base-case} implies that
\begin{equation}\label{use-bc}
\int_{\Delta_*^{i,l}} \Big( \textbf{k}_*^{X_{\Delta_*^{i,l}}} \Big)^p d\sigma_*
\lesssim
\sigma_*(\Delta_*^{i,l})^{1-p}.
\end{equation}
where $X_{\Delta_*^{i,l}}$ is a corkscrew point relative to $\Delta_*^{i,l}$ in $\Omega_*$. Using Lemma \ref{lm:cop} applied with $\mathcal{D}=\Omega_*$ and Lebesgue's differentiation theorem  we have that $\textbf{k}_*^{X}(y)\approx \omega_*^X(\Delta_*^{i,l})\,\textbf{k}_*^{X_{\Delta_*^{i,l}}}(y)$ for $\sigma_*$-a.e. $y\in \Delta_*^{i,l}$. As a result, using \eqref{use-bc}
\begin{multline}
\int_{\Delta_*^{i,l}} \big( \textbf{k}_*^X \big)^p d\sigma_*
\approx
\big(\omega^X_*(\Delta_*^{i,l})\big)^p
\int_{\Delta_*^{i,l}} \Big( \textbf{k}_*^{X_{\Delta_*^{i,l}}} \Big)^p d\sigma_*
\lesssim
\big(\omega^X_*(\Delta_*^{i,l})\big)^p\,\sigma_*(\Delta_*^{i,l})^{1-p}
\\
=\sigma_*(\Delta_*^{i,l}) \left( \frac{\omega_*^X(\Delta_*^{i,l})}{\sigma_*(\Delta_*^{i,l}) } \right)^p \lesssim
\sigma_*(\Delta_*^i) \left( \frac{\omega_*^X(\Delta_*^{i})}{\sigma_*(\Delta_*^{i}) } \right)^p \label{eq:KP},
\end{multline}
where we used the Ahlfors regularity of $\sigma_\ast$ and the doubling properties of $\omega_\ast$.
We claim that  
	\begin{equation}\label{eq:cohm}
\frac{\omega_*^X(\Delta_*^{i})}{\sigma_*(\Delta_*^{i}) } \lesssim \frac{ \omega^X(Q^i) }{\sigma(Q^i)}.
\end{equation}
To see this write $u_1(Y)=\omega^Y_*(\Delta_*^{i} )$ and $u_2(Y)=\omega^Y(Q^i)$ for every $Y\in \Omega_*$ and note that $L u_1=L u_2=0$ in $\Omega_*\subset\Omega$. For $Y\in \Delta_*^{i}\subset \overline{\Omega_*}\subset \Omega$ we have $u_2(Y)\gtrsim 1$ by Lemma \ref{lm:Bourgain} applied in $\mathcal{D}=\Omega$, Harnack's inequality, \eqref{eq:sizeSi}, and \eqref{def:WQ}. Thus the maximum principle applied in the bounded open set $\Omega_*$ yields that $u_1(Y)\lesssim u_2(Y)$ for every $Y\in\Omega_*$, hence in particular for $Y=X$. This and the fact that $\pom$ and $\pom_*$ are Ahlfors regular (see Lemma \ref{lemma:sawtooth-inherit}) give  at desired \eqref{eq:cohm}.

Combining \eqref{eq:KP} and \eqref{eq:cohm}, and using Hölder's inequality and Ahlfors regularity of $\sigma, \sigma_*$, we get
	\begin{align}
		\int_{\Delta_*^{i,l}} \left( \textbf{k}_*^X \right)^p d\sigma_* \lesssim \sigma(Q^i) \left( \frac{ \omega^X(Q^i) }{\sigma(Q^i)} \right)^p \lesssim \int_{Q^i} \left(\textbf{k}^X \right)^p d\sigma.
	\end{align}
We recall that $S_*^i$ is covered by a uniformly bounded number of surface balls $\Delta_*^{i,l}$. Thus summing in $l$ we conclude \eqref{eq:Si} as desired. This completes \textbf{Step 1}.

\noindent\textbf{Step 2.} Study the interaction of the elements of the family $\{Q^i: i\in \N_{\Delta_*}\}$. 
	
We first note that	for every $i\in \N_{\Delta_*}$
\begin{equation}\label{eq5.24}
 \dist(\Delta_*, Q) \leq \dist(S_*^i\cap \Delta_*,Q) \lesssim \ell(Q^i)\approx \ell(I^i)\lesssim r_* 
\end{equation}
Pick $\hat x\in \overline Q$ such that $\dist(\hat x, \Delta_*) = \dist(Q, \Delta_*)$. If $\hat{x}\in \overline Q \setminus Q$, we 
	replace it by a point, which we call again $\hat x$, belonging to $B(\hat x, r_*/2) \cap Q$, so that $\hat x \in Q$ and $\dist(\hat x, \Delta_*) \lesssim r_*$. We claim that 
	there is a large constant $C>1$ such that $Q^i \subset \Delta_1$ where $\Delta_1:=B(\hat x, Cr_*)\cap\pom$. Indeed if $y\in Q^i$ then
	\[
	|y-\hat x|\le \diam(Q^i)+\dist(Q^i, I^i)+\diam(I^i)+|y^i-\hat x|\lesssim r_*,
	\]
	where we have picked $y^i \in S_*^i \cap \Delta_*$ for each $i\in\mathcal{N}_{\Delta_*}$.

Consider next the covering 
$\Delta_1\subset \cup_{k=1}^{N_1} P_k$, where $N_1$ depends on Ahlfors regularity and dimension, and $\{P_k\}_{k=1}^{N_1}$ is a pairwise disjoint collection of dyadic
	cubes on $\pO$, of the same generation, with length $\ell(P_k)\approx r_*$.  Since in the present scenario,
$\ell(Q^i) \lesssim r_*$,
we may further suppose that $\ell(P_k) \geq\ell(Q^i)$ for every $i$.   Moreover, since we have assumed that  $r_* \leq M_1^{-1} \ell(Q)$, taking $M_1$ large enough we may assume that $\ell(P_k)\le \ell(Q)$ for every $1\le k\le N_1$. 

Note that 
\[ \bigcup_{i\in \N_{\Delta_*}} Q^i \subset \Delta_1 \subset \bigcup_{k=1}^{N_1} P_k.\]
By relabeling if needed, we may assume that there exists $N_2$,  $1\le N_2\le N_1$, such that $P_k$ meets some $Q^i$, $i\in \N_{\Delta_*}$, for each $1\le k\le N_2$. 
Hence $\bigcup_{i\in \N_{\Delta_*}} Q^i \subset \bigcup_{k=2}^{N_2} P_k$ and, necessarily, $Q^i\subset P_k\subset Q$, and   since $Q^i \in \DD_{\F,Q}$, it follows that $P_k\in \DD_{\F,Q}$ for $1\le k\le N_2$.

For future reference, we record the following observation.  Recall that $X$ is a Corkscrew point relative to
$\Delta_*=B_*\cap\pO_*$, for the domain $\Omega_*$; i.e.,  $X\in B_*\cap\Omega_*$, with
$\delta_*(X) \approx r_*$.  By \eqref{eq5.24} and for every $1\le k\le N_2$ if we pick some $i$ so that $Q^i\subset P_k$ we have
\[
r_*
\approx 
\delta_*(X) 
\le
\delta(X) 
\le  
\dist(X,P_k)
\le 
\dist(X, Q^i)
\le
|X-x_*|+2\,r_*+\dist(\Delta_*, Q_i)
\lesssim
r_*
\approx \ell(P_k).
\]
Recalling that $X_{P_k}$ denotes a corkscrew point relative to the dyadic cube $P_k$ we then have that $\delta(X)\approx \ell(P_k)\approx \delta(X_{P_k})$ and also $|X-X_{P_k}|\lesssim \ell(P_k)$, hence by Harnack's inequality $\omega^X\approx\omega^{X_{P_k}}$ and eventually $\textbf{k}^X\approx \textbf{k}^{X_{P_k}}$, $\sigma$-a.e. in $\pom$. On the other hand, we have already mentioned that hypothesis \ref{1-thm:trsf} in Theorem \ref{thm:trsf} says that $\omega\in RH_{p_1}(\sigma)$, which clearly implies $\omega\in RH_{p}(\sigma)$ since $p\le p_1$. Note that this reverse Hölder condition is written for surface balls, but it is straightforward to see, using Lemmas \ref{lm:ddAR} and \ref{lm:doubling}, that the same reverse Hölder estimates hold for any dyadic cube. All these,  and the fact that both $\pom$ and $\pom_*$ are Ahlfors regular (see Lemma \ref{lemma:sawtooth-inherit}) lead to
\begin{equation}\label{eq5.25}
\int_{P_k} \Big(\textbf{k}^{X_{P_k}}\Big)^{p} d\sigma
\lesssim 
\sigma(P_k)\bigg(\frac{\omega^{{X_{P_k}}}(P_k)}{\sigma(P_k)}\bigg)^p
\le 
\sigma(P_k)^{1-p} \approx \sigma_*(\Delta_*)^{1-p}\,,
 \end{equation}
for each $k$, with uniform implicit constants.

As mentioned above, for every $i\in\N_*$, there exists $J^i\in\W$ so that $I^i\cap J^i\neq\emptyset$ and so that if we pick $\widetilde{Q}^i\in\dd$ with $\ell(\widetilde{Q}^i)=\ell(J^i)$ and $\dist(J,\pom)=\dist(J^i, \widetilde{Q}^i)$ then $\widetilde{Q}^i\notin \DD_{\F,Q}$.  In particular 
	\begin{equation}\label{bdQngh}
		\ell(\widetilde{Q}^i) \approx \ell(Q^i)\quad \text{ and }\quad \dist(\widetilde{Q}^i, Q^i) \lesssim \ell(Q^i). 
	\end{equation}  
By the definition of $\DD_{\F,Q}$, $\widetilde{Q}^i \notin \DD_{\F,Q}$ means either $\widetilde{Q}^i \subset \pom\setminus Q$, or $\widetilde{Q}^i \subset Q_j $, for some
	$Q_j\in \F$. Given $1 \le k\le N_2$, for each $i\in \N_{\Delta_*}$, we say $i\in \N_0(k)$, 
if the first case happens, with $Q^i\subset P_k$; 
and if the second case happens with $Q_j\in \F$, and with $Q^i\subset P_k$,
we say $i\in \N_j(k)$.
	For the second case we remark that
	\begin{equation}\label{eq:Qjclose}
		\dist(Q_j, P_k) \leq \dist(\widetilde{Q}^i, Q^i) \lesssim \ell(Q^i) \leq \ell(P_k).
	\end{equation}
For each $k$, $1\le k\le N_2$, we set
\[
\F_1(k):=\{Q_j\in\F: \exists\, i\in\mathcal{N}_j(k),\ \ell(Q_j) \geq \ell(P_k)\}
\]
and
\[
\F_2(k):=\{Q_j\in\F: \exists\, i\in\mathcal{N}_j(k),\ \ell(Q_j) <\ell(P_k)\}.
\]

With the previous notation, \eqref{eq:tosumN}, and the fact that $\bigcup_{i\in \N_{\Delta_*}} Q^i \subset \bigcup_{k=2}^{N_2} P_k$
we obtain
\begin{align}\label{eq:tosumN:cont}
\int_{\Delta_*} \left(\textbf{k}_*^X \right)^p d\sigma_* 
&\lesssim \sum_{i\in \N_{\Delta_*}} \int_{Q^i} \left(\textbf{k}^X \right)^p d\sigma
\\
&\le
\sum_{k=1}^{N_2} \Bigg(\sum_{i\in \N_0(k)}\int_{Q^i} \left(\textbf{k}^X \right)^p d\sigma+\sum_{Q_j\in\F}\sum_{i\in \N_j(k)}\int_{Q^i} \left(\textbf{k}^X \right)^p d\sigma
\bigg)
\nonumber \\
&\lesssim
\sum_{k=1}^{N_2} \Bigg(\sum_{i\in \N_0(k)}\int_{Q^i} \left(\textbf{k}^{X_{P_k}} \right)^p d\sigma+\sum_{Q_j\in\F_1(k)}\sum_{i\in \N_j(k)}\int_{Q^i} \left(\textbf{k}^{X_{P_k}} 
\right)^p d\sigma 
\nonumber \\
&\hskip3cm
+ \sum_{Q_j\in\F_2(k)}\sum_{i\in \N_j(k)}\int_{Q^i} \left(\textbf{k}^{X_{P_k}} 
\right)^p d\sigma\bigg)
\nonumber ,
\end{align}
where we have used that $\textbf{k}^X\approx \textbf{k}^{X_{P_k}}$, $\sigma$-a.e. in $\pom$. 
	
	At this stage we need the following lemma. We defer its proof until later.
	\begin{lemma}\label{lm:thinbd} 
		Let $\mathcal{D}$ be an open set with Ahlfors regular boundary and write $\sigma=\HH^{n-1}|_{\partial\mathcal{D}}$.
		Let $Q\in \dd=\DD(\partial\mathcal{D})$ and suppose that $\DD'\subset\DD$ is such that each 
$Q'\in \DD'$ satisfies one of the following conditions for some $C_1\ge 1$:
		\begin{itemize}\itemsep=0.2cm
			\item $Q' \subset Q$ and $\dist(Q',\pom\setminus Q) \leq C_1 \ell(Q')$.
			\item $Q' \cap Q = \emptyset$, $\ell(Q') \leq C_1 \ell(Q)$ and $\dist(Q', Q) \leq C_1 \ell(Q')$.
		\end{itemize} 
Then there is a subcollection of distinct cubes 
$\{\widetilde{Q}_m\}_{m=1}^{N_2}$, all of the same generation, with $N_2 = N_2(n,C_{AR}, C_1)$, 
satisfying $\ell(Q)\le \ell(\widetilde Q_m) \le C_2 \ell(Q)$ and $\dist(\widetilde{Q}_m,Q) \le  C_2\ell(Q)$, 
with $C_2 = C_2(n,C_{AR}, C_1)$, for every $m$, such that for any $s>1$ if $0\le h\in L_{\rm loc}^s(\partial\mathcal{D},\sigma)$ then 
		\begin{equation}\label{ccl:thinbd}
			\sum_{Q' \in \DD'} \int_{Q'} h d\sigma \le C_3 \sigma(Q) 
\sum_{m=1}^{N_2} \left( \fint_{ \widetilde Q_m} h^s d\sigma \right)^{\frac{1}{s}}
		\end{equation}
		where $C_3=C_3(n,C_{AR}, C_1,s)$.

As a consequence, if there exists $C_1'$ so that for each $m$, $1\le m\le N_2$, there holds
\begin{equation}\label{RH-h}
\left( \fint_{ \widetilde Q_m} h^s d\sigma \right)^{\frac{1}{s}}\le C_1' \fint_{ \widetilde Q_m} h\, d\sigma
\end{equation}
then
\begin{equation}\label{ccl:thinbd:Ch}
\sum_{Q' \in \DD'} \int_{Q'} h d\sigma \le C_3' 
\sum_{m=1}^{N_2} \int_{\widetilde Q_m } h d\sigma 
\end{equation}
with $C_3'=C_3'(n,C_{AR}, C_1,s, C_1')$.
\end{lemma}

	\begin{remark} 
It follows from the proof of that if $Q'\subset Q$ for all $Q'\in \DD'$ (i.e., we only consider the first case), then 
there is only one $\widetilde{Q}_m$, namely the unique one containing $Q$ satisfying the given conditions. 	
	\end{remark}

\begin{remark}\label{rem:daydic-lemma:omega}
Suppose that we are under the assumptions of the previous result. Assume further that $\mathcal{D}$ is a uniform domain with Ahlfors regular boundary and that $\omega_L\in RH_p(\sigma)$. Then, if $\textbf{k}_L=d\omega_L/d\sigma$ it follows that 
\begin{equation}
\label{ccl:thinbd2}
\sum_{Q' \in \DD'} \int_{Q'} \left(\textbf{k}_L^{X_Q}\right)^p d\sigma \lesssim \int_{Q} \left( \textbf{k}^{X_Q} \right)^p d\sigma.
\end{equation}
with an implicit constant depending on the allowable constants of $\mathcal{D}$, $C_1$, $p$, and the implicit constant in the condition  $\omega_L\in RH_p(\sigma)$.

To see this we recall that from Gehring's Lemma it follows that there exists $s>1$ such that  $\omega_L\in RH_{ps}(\sigma)$. This, combined with Harnack's inequality, implies that \eqref{RH-h} holds with $h=\big(\textbf{k}_L^{X_Q}\big)^p$. As a result \eqref{ccl:thinbd:Ch} readily gives \eqref{ccl:thinbd2}: 
\begin{multline*}
\sum_{Q' \in \DD'} \int_{Q'}\left(\textbf{k}_L^{X_Q}\right)^pd\sigma \lesssim 
\sum_{m=1}^{N_2} \int_{\widetilde Q_m } \left(\textbf{k}_L^{X_Q}\right)^p d\sigma 
\approx
\sum_{m=1}^{N_2} \int_{\widetilde Q_m } \left(\textbf{k}_L^{X_{\widetilde Q_m}}\right)^p d\sigma 
\\
\lesssim
\sum_{m=1}^{N_2} \sigma(\widetilde Q_m)^{1-p}
\lesssim
\sigma(Q)^{1-p},
\end{multline*}		
where we have used Harnack's inequality (to change the pole of the elliptic measure from $X_Q$ to $X_{\widetilde Q_m}$ and the fact that $N_2$ is uniformly bounded). 
\end{remark}
	
We will use the previous remark to estimate \eqref{eq:tosumN:cont}. Fixed then $1\le k\le N_2$ and we split the proof in three different steps.

\noindent\textbf{Step 2.1.} Estimate for $\N_0(k)$. 

If $i\in \N_0(k)$ we have $\widetilde{Q}^i \subset \pom\setminus Q \subset \pom\setminus P_k$ and 
	\[ \dist(Q^i, \pom\setminus P_k) \leq \dist(Q^i, \widetilde{Q}^i) \lesssim \ell(Q^i). \]
	Since $Q^i \subset P_k$,  we may apply Lemma \ref{lm:thinbd} to $P_k$ 
	and the collection $\DD' := \{Q^i: i\in \N_0(k)\}$ (note that we are in the first scenario), to obtain by Remark \ref{rem:daydic-lemma:omega}
\begin{equation}\label{eq:N0}
\sum_{i\in \N_0(k)} 
\int_{Q^i} \left(\textbf{k}^{X_{P_k}}\right)^p d\sigma
\lesssim 
\int_{P_k} \left(\textbf{k}^{X_{P_k}} \right)^p d\sigma  
\lesssim \sigma_*(\Delta_*)^{1-p},
	\end{equation}
 where in the last inequality we have used \eqref{eq5.25}.

\noindent\textbf{Step 2.2.} Estimate for $Q_j\in \F_1(k)$.

	By \eqref{eq:Qjclose}, the cardinality of $\F_1(k)$ is uniformly bounded. Moreover, for each $Q_j\in \F_1(k)$ we necessarily have $Q_j \cap P_k = \emptyset$, since otherwise, the 
condition $\ell(Q_j) \geq \ell(P_k)$ 
guarantees that $P_k\subset Q_j$, and thus $Q^i \subset P_k \subset Q_j\in\F$. This 
contradicts that $Q^i \in \DD_{\F,Q}$. On the other hand $Q_j \cap P_k = \emptyset$ implies 
$\widetilde{Q}^i\subset Q_j\subset \pom\setminus P_k $, for each $i\in \N_j(k)$. 
Combined with \eqref{bdQngh}, this yields
	\[ \dist(Q^i, \pom\setminus P_k) \leq \dist(Q^i, \widetilde{Q}^i) \lesssim \ell(Q^i). \]

	Applying Lemma \ref{lm:thinbd} to $P_k$
	and the collection $\DD' = \{Q^i: i\in \N_j(k)\} $ (note that we are in the first scenario), we obtain from \eqref{ccl:thinbd2}
\begin{equation}\label{eq:Nj1}
		\sum_{i\in \N_j(k)} \int_{Q^i} \left(\textbf{k}^{X_{P_k}} \right)^p d\sigma \lesssim 
\int_{P_k} \left(\textbf{k}^{X_{P_k}} \right)^p d\sigma 
\lesssim \sigma_*(\Delta_*)^{1-p}, 
	\end{equation}
 where again we have used \eqref{eq5.25}.  
	The above estimate holds for each $Q_j\in \F_1(k)$, which as uniformly bounded cardinality, hence
	\begin{equation}\label{eq:Nj1all}
		\sum_{Q_j\in \F_1(k)} \sum_{i\in \N_j(k)} \int_{Q^i} \left(\textbf{k}^{X_{P_k}}  \right)^p d\sigma  
		\lesssim \sigma_*(\Delta_*)^{1-p}. 
	\end{equation}

\noindent\textbf{Step 2.3.} Estimate for $Q_j\in \F_2(k)$. 

For each $Q_j\in \F_2(k)$ we claim that
	\begin{equation}\label{eq:Nj2}
		\sum_{i\in \N_j(k)} 
		\int_{Q^i} \left(\textbf{k}^{X_{P_k}} \right)^p d\sigma \lesssim \int_{Q_j} \left( \textbf{k}^{X_{P_k}}  \right)^p d\sigma.
	\end{equation}
	In fact, for each $i\in \N_j(k)$, by \eqref{bdQngh} and $\widetilde{Q}^i \subset Q_j$, we have
	\begin{equation}\label{sizeupbd}
		\ell(Q^i) \approx \ell(\widetilde{Q}^i) \leq \ell(Q_j).
	\end{equation}
	Since $Q^i \in \DD_{\F,Q}$, we either have $Q^i \cap Q_j = \emptyset$, or $Q_j \subsetneq Q^i$. 
In the first case, note that
\[
\dist(Q^i,Q_j) \leq \dist(Q^i,\widetilde{Q}^i))\lesssim\ell(Q^i),
\]
hence $Q^i\cup Q_j\subset \Delta(x_{Q_j}, C\,\ell(Q_j))$ which $x_{Q_j}$ being the center of $Q_j$ an a uniform constant $C$. By Lemma \ref{lm:cop} applied with $\mathcal{D}=\Omega$ (or Harnack's inequality if $\ell(Q_j)\approx \ell(P_k)$), Lebesgue's differentiation theorem, Lemma \ref{lm:doubling}, and Harnack's inequality one can see that
\[
\textbf{k}^{X_{P_k}}(y)\approx \omega^{X_{P_k}}(Q_j)\, \textbf{k}^{X_{Q_j}}(y),
\qquad
\text{for $\sigma$-a.e. $y\in  \Delta(x_{Q_j}, C\,\ell(Q_j))$}.
\]
This, Lemma \ref{lm:thinbd} with $Q_j$ and 
the collection $\DD': = \{Q^i: i\in \N_j(k), Q^i \cap Q_j = \emptyset \}$ (we are in the second scenario), and Remark \ref{rem:daydic-lemma:omega} lead to
\begin{multline}\label{eq:Nj2out}
	\sum_{\substack{i\in \N_j(k) \\ Q^i \cap Q_j = \varnothing} } \int_{Q^i} \left(\textbf{k}^{X_{P_k}} \right)^p d\sigma 
	\approx
		\left(\omega^{X_{P_k}}(Q_j)\right)^p \sum_{\substack{i\in \N_j(k) \\ Q^i \cap Q_j = \varnothing} } \int_{Q^i} \left(\textbf{k}^{X_{Q_j}} \right)^p d\sigma 
		\\
		\lesssim \left(\omega^{X_{P_k}}(Q_j) \right)^p\int_{Q_j} \left( \textbf{k}^{X_{Q_j}} \right)^p d\sigma
		\approx
		\int_{Q_j} \left( \textbf{k}^{X_{P_k}} \right)^p d\sigma
	.
\end{multline}

On the other hand, if $Q_j \subsetneq Q^i$, then \eqref{sizeupbd} gives $\ell(Q_j)\approx \ell(Q^i)$, hence the cardinality of $\{Q^i: i\in \N_j(k), Q_j \subsetneq Q^i\}$ is uniformly bounded. On the other hand, by Lemma \ref{lm:cop} applied with $\mathcal{D}=\Omega$ (or Harnack's inequality if $\ell(Q^i)\approx \ell(P_k)$), Lebesgue's differentiation theorem, Lemma \ref{lm:doubling}, and Harnack's inequality we readily obtain 
\[
\textbf{k}^{X_{P_k}}(y)\approx \omega^{X_{P_k}}(Q_j)\, \textbf{k}^{X_{Q_j}}(y),
\qquad
\text{for $\sigma$-a.e. $y\in  Q^i$}.
\]
Thus, using  Corollary \ref{cor:doublingPk} we have
\begin{align}\label{eq:Nj2in}
\sum_{\substack{i\in \N_j(k) \\ Q^i \supsetneq Q_j}} \int_{Q^i} \left(\textbf{k}^{X_{P_k}} \right)^p d\sigma 
& \approx
\left ( \omega^{X_{P_k}}(Q_j)\right)^p\,\sum_{\substack{i\in \N_j(k) \\ Q^i \supsetneq Q_j}}  \int_{Q^i} \left(\textbf{k}^{X_{Q_j}} \right)^p d\sigma 
\\ \nonumber
 &\lesssim
\left ( \omega^{X_{P_k}}(Q_j)\right)^p\, \sum_{\substack{i\in \N_j(k) \\ Q^i \supsetneq Q_j}} \int_{Q_j} \left(\textbf{k}^{X_{Q_j}} \right)^p d\sigma 
\\ \nonumber
 &\lesssim 
\left ( \omega^{X_{P_k}}(Q_j)\right)^p\, \int_{Q_j} \left(\textbf{k}^{X_{Q_j}} \right)^p d\sigma 
\\ \nonumber
 &\approx
 \int_{Q_j} \left(\textbf{k}^{X_{P_k}} \right)^p d\sigma.
\end{align}
The claim \eqref{eq:Nj2} now follows from \eqref{eq:Nj2out} and \eqref{eq:Nj2in}. 
	
To continue, let us recall that for each $Q_j\in \F_2(k)$, 
	\[ \ell(Q_j) < \ell(P_k)\quad \text{ and }\quad  \dist(Q_j,P_k) \lesssim \ell(P_k), \]
	where the second inequality is \eqref{eq:Qjclose}.
Consequently, each $Q_j\in \F_2(k)$, is contained in some $P\in
	{\bf N}(P_k):= \{P\in\dd: \ell(P) =\ell(P_k),\  \dist(P,P_k) \lesssim \ell(P_k)\}$ and, clearly, the cardinality of ${\bf N}(P_k)$ is uniformly bounded.
Recalling that $\F=\{Q_j\}_j$ is a pairwise disjoint family of cubes, 
by \eqref{eq:Nj2}, Corollary \ref{cor:doublingPk}, and \eqref{eq5.25}, we arrive at
\begin{multline}
		\sum_{Q_j\in \F_2(k)} \sum_{i\in \N_j(k)} 
		\int_{Q^i} \left(\textbf{k}^{X_{P_k}} \right)^p d\sigma \lesssim 
		\sum_{Q_j\in \F_2(k)} \int_{Q_j} \left(\textbf{k}^{X_{P_k}} \right)^p d\sigma \\[4pt] = \,
		\int_{\bigcup\limits_{Q_j\in \F_2(k)} Q_j} \left(\textbf{k}^{X_{P_k}} \right)^p d\sigma 
		 \le \sum_{P\in {\bf N}(P_k)}\int_{P} \left(\textbf{k}^{X_{P_k}} \right)^p d\sigma  
		\lesssim \int_{P_k} \left(\textbf{k}^{X_{P_k}} \right)^p d\sigma \label{eq:Nj2all}
\lesssim\sigma_*(\Delta_*)^{1-p}.
	\end{multline}

\noindent\textbf{Step 2.4.} Final estimate.

We finally combine \eqref{eq:tosumN:cont} with \eqref{eq:N0}, \eqref{eq:Nj1all}, and \eqref{eq:Nj2all}, and use the fact that $N_2\le N_1=N_1(n, C_{AR})$  
to conclude that 
	\begin{align}
\int_{\Delta_*} \left(\textbf{k}_*^X \right)^p d\sigma_*
\lesssim
\sum_{k=1}^{N_2} \sigma_*(\Delta_*)^{1-p}
\lesssim \sigma_*(\Delta_*)^{1-p}.
\label{eq:sumallN}
	\end{align} 
Hence, we have obtained the desired estimate \eqref{RH},
and therefore the proof of Theorem \ref{thm:trsf} is complete, 
provided that the sawtooth domain $\Omega_\ast$ is \textbf{compactly contained} in $\Omega$ and modulo the proof of Lemma \ref{lm:thinbd}.

To consider the general case we need the following theorem which generalizes \cite[Theorem 4.1]{KKiPT} and \cite{DJK} (see also \cite{DKP, Zh}):
\begin{theorem}[{\cite[Theorem 1.1]{CHMT}}]\label{thm:CMEAinfty}
	Let $\mathcal{D}$ be uniform domain $\mathcal{D}$ Ahlfors regular boundary, and let $\wcalA$ be a real (non necessarily symmetric) uniformly elliptic matrix on  $\mathcal{D}$. The following are equivalent:
	\begin{enumerate}[label=\textup{(\arabic*)}, itemsep=0.2cm] 	
		\item\label{1-thm:CMEAinfty}  The elliptic measure $\omega_L$ associated with the operator $L=-\divg(\wcalA\nabla)$ is of class $A_\infty$ with respect to the surface measure. 
		
		\item\label{2-thm:CMEAinfty} Any bounded weak solution to $Lu=0$ satisfies the Carleson measure estimate 
		\begin{equation}
		\sup_{\substack{x\in\partial\mathcal{D} \\ 0<r<\infty}} \frac1{r^n}\iint_{B(x,r) \cap \mathcal{D}} |\nabla u(Y)|^2 \dist(Y,\partial\mathcal{D})\, dY \leq C\|u\|_{L^\infty(\mathcal{D})}^2.
		\end{equation}
	\end{enumerate}
The involved constants depend on the allowable constants and the constant appearing in the corresponding hypothesis of the implication in question. 
\end{theorem} 

Consider next a \textbf{general} sawtooth domain $\Omega_* = \Omega_{\F,Q}$ which, although bounded, is not necessarily compactly contained in $\Omega$. By \ref{2-thm:CMEAinfty} $\implies$ \ref{1-thm:CMEAinfty} in Theorem \ref{thm:CMEAinfty} with $\mathcal{D}=\Omega$, in order to obtain that the elliptic measure associated with $L$ relative to $\Omega_*$ belongs to $A_\infty$ with respect to the surface measure, we just need to see that \ref{2-thm:CMEAinfty} holds with $\mathcal{D}=\Omega_*$. With this goal in mind we take $u$, a bounded weak solution to $Lu=0$ in $\Omega_*$, and let $x\in\pom$ and $0<r<\infty$.

Given $N\ge 1$ we recall the definition of $\F_N:=\F(2^{-N}\,\ell(Q))$ in Section \ref{section:sawtooth}  and write $\Omega_*^N := \Omega_{\F_N,Q}$. Note that by construction $\ell(Q')> 2^{-N}\,\ell(Q)$ for every $Q'\in\dd_{\F_N, Q}$, hence $\Omega_*^N$ is compactly contained in $\Omega$ (indeed is at distance of the order $2^{-N}\,\ell(Q)$ to $\pom$). Then we can apply the previous case to obtain that for each $N$, the associated elliptic measure associated with $L$ relative to $\Omega_*^N$ satisfies the $A_\infty$ property with respect to the surface measure of $\partial\Omega_\ast^N$, and the implicit constants depend only on the allowable constants. Hence 
\ref{1-thm:CMEAinfty} $\implies$ \ref{2-thm:CMEAinfty} in Theorem \ref{thm:CMEAinfty} with $\mathcal{D}=\Omega_*^N$ implies
\begin{equation}\label{CME-N}
\sup_{\substack{z\in\pom_*^N \\ 0<s<\infty}} \frac1{s^n}\iint_{B(z,s) \cap \Omega_*^N} |\nabla u(Y)|^2 \delta_*^N(Y)\, dY \lesssim \|u\|_{L^\infty(\Omega_N^*)}^2
\le
C\|u\|_{L^\infty(\Omega_*)}^2,
\end{equation}
since $u$ is a bounded weak solution to $Lu=0$ in $\Omega_*$ and so in each $\Omega_N^*$,  where $\delta_*^N=\dist(\cdot\,,\pom_*^N)$ and where the implicit constants depend only on the allowable constants.

Let $\omega_*$ and $\omega_*^N$ denote the elliptic measures to $L$ relative to $\Omega_*$ and $\Omega_*^N$ respectively, and let $\sigma_*^N := \mathcal{H}^{n-1}|_{\pO_*^N}$ denote the surface measure of $\partial\Omega_\ast^N$. By construction $\{\Omega_*^N\}_{N\ge 1}$ is an increasing sequence of sets with $\Omega_\ast=\cup_{N\ge 1}\Omega_*^N$. Hence, for any $Y\in \Omega_*$ there is $N_Y\ge 1$ such that $Y\in\Omega_\ast^N$ for all $N\ge N_Y$.  Clearly $\delta_*^N(Y) \nearrow \delta_*(Y)$ as $N\to\infty$ and
\[ 
|\nabla u(Y)|^2 \delta_*^N(Y) \chi_{\Omega_*^N}(Y) \nearrow |\nabla u(Y)|^2 \delta_*(Y) \chi_{\Omega_*}(Y), \quad \text{ as } N\to \infty.
\]
On the other hand since $x\in \pO_*$, using the Corkscrew condition we can find a sequence $\{x_N \}_{N\ge 1}$ with $x_N\in \pO_*^N$ such that $x_N \to x$. In particular, $B(x,r) \subset B(x_N,2r)$ for sufficiently large $N$. By Fatou's Lemma and \eqref{CME-N} it then follows
	\begin{multline}\label{tt-2}
		\iint_{B(x,r) \cap \Omega_*} |\nabla u(Y)|^2 \delta_*(Y)\, dY 
		\leq \liminf_{N\to\infty} \iint_{B(x,r) \cap \Omega_*^N} |\nabla u(Y)|^2 \delta_*^N(Y)\, dY 
		\\
			\leq \liminf_{N\to\infty} \iint_{B(x_N,2r) \cap \Omega_*^N} |\nabla u|^2 \delta_*^N(y)\, dY
			\lesssim
			r^n\,
			\|u\|_{L^\infty(\Omega_*)}^2,
		\end{multline}
		where the implicit constant depend only on the allowable constants. Since $x$, $r$, and $u$ are arbitrary we have obtained as desired \ref{2-thm:CMEAinfty} in Theorem \ref{thm:CMEAinfty} for $\mathcal{D}=\Omega_*$ and as a result we conclude that $\omega_*\in A_\infty(\sigma_*)$. This completes the proof for an arbitrary sawtooth domain $\Omega_*$, and therefore the proof of Theorem \ref{thm:trsf} modulo the proof of Lemma \ref{lm:thinbd}.
	\end{proof}

	\begin{proof}[Proof of Lemma \ref{lm:thinbd}]
For fixed $k\in \ZZ$, write $\DD'_k:=\{Q'\in\dd': \ell(Q') = 2^{-k} \ell(Q)\}$, which is a pairwise disjoint family. In the first case since $Q'\subset Q$, we have that  $k\ge 0$; in the second case since $\ell(Q') \leq C_2 \ell(Q)$, we may assume that $k\ge -\log_2 C_2$. Set then $k_0=0$ in the first case and $k_0$ the integer part of $\log_2 C_1$. We define for $k\ge -k_0$
		\[ A_k^+ = \{x\in Q: \dist(x,\pom\setminus Q) \lesssim 2^{-k} \ell(Q) \}, \ \ \quad A_k^- = \{x\in\partial\Omega\setminus Q: \dist(x,Q) \lesssim 2^{-k} \ell(Q) \}, \]
		so that for
		appropriate choices of the implicit constants, 
each $Q'\in \DD'_k$ is contained in either $A_k^+$ (the first case) or $A_k^-$ (the second case).
		Recall that by the thin boundary property of the 
dyadic decomposition $\DD$ (cf. \eqref{thin-boundary}), there is $\gamma \in (0,1)$ such that for all $k$ under consideration,
		\[ 
			\sigma(A_k^+) \lesssim 2^{-k\gamma} \sigma(Q), \quad \sigma(A_k^-) \lesssim 2^{-k\gamma}\sigma(Q). 
		\]
Set 
$$\F_-:= \big\{Q'\subset \pom\setminus Q: \ell(Q') \leq C_1 \ell(Q),\ \dist(Q',Q) \leq C_1 \ell(Q')\big\}.$$
Observe that each $Q'\in \F_-$ is contained in some dyadic cube $\widetilde{Q}$, with
 $\ell(\widetilde Q)\approx \ell(Q)$ and $\dist(\widetilde{Q},Q) \lesssim \ell(Q)$  depending on $C_1$.
We may therefore define a collection of distinct cubes 
$\F_*:= \{\widetilde{Q}_m\}_{m=1}^{N_2}$, all of the same dyadic generation, 
one of which (say, $\widetilde{Q}_1$) contains $Q$, with $\ell(\widetilde Q_m) \approx \ell(Q)$, and with
$\dist(\widetilde{Q}_m,Q) \lesssim \ell(Q)$  for every $m$,
such that each $Q'\in\F_-$ is contained in some $\widetilde{Q}_m\in\F_*$, and 
$$\bigcup_k A_k^+ \,\subset\, Q\,\subset \,\widetilde{Q}_1\,, \quad {\rm and} \quad 
\bigcup_k A_k^- \subset \, \bigcup_{m=2}^N \widetilde{Q}_m.$$
Clearly, we have $\#\F_* = N_2= N_2(n,C_{AR}, C_1)$. Using all the previous observations we get for any $s>1$
\begin{align}\label{tt-100}
		\sum_{Q' \in \DD'} \int_{Q'} h\, d\sigma 
		&=
			 \sum_{k=-k_0}^\infty \sum_{Q' \in \DD'_k} \int_{Q'} h\, d\sigma 
\\
&			\leq  \sum_{k=-k_0}^\infty \int_{A_k^+ \cup A_k^-} h\, d\sigma
			\nonumber\\ 
			& \leq \sum_{k=-k_0}^\infty \sigma(A_k^+ \cup A_k^-)^{1-\frac1s} 
			\left( \int_{\cup_m\widetilde Q_m} h^s d\sigma \right)^{\frac{1}{s}}
			\nonumber\\
			& \lesssim \sum_{k=-k_0}^\infty \left(2^{-k\gamma }\sigma(Q) \right)^{1-\frac{1}{s}} 
	 \left( \int_{\cup_m\widetilde {Q}_m} h^s\, d\sigma \right)^{\frac{1}{s}} 
	\nonumber \\
			& \lesssim \sigma(Q) \left(\sum_m \fint_{\widetilde Q_m} h^s\, d\sigma \right)^{\frac{1}{s}}
	\nonumber\\
		&	\lesssim \sigma(Q) \sum_m 
		\left( \fint_{\widetilde Q_m} h^s d\sigma \right)^{\frac{1}{s}}.\nonumber
		\end{align}
This shows \eqref{ccl:thinbd}. To obtain \eqref{ccl:thinbd:Ch} we combine \eqref{tt-100} together with \eqref{RH-h} and the fact that $\sigma(Q)\approx\sigma(Q_m)$ for every $1\le m\le N_2$ by the Ahlfors regular property and the construction of the family $\F_*$. 
	\end{proof}

\section{Optimality}\label{optimal}

As we mentioned in the introduction, the class of elliptic operators we consider is optimal to guarantee the $A_\infty$ property. In this section we illustrate the optimality from two different points of view. See Proposition \ref{p:bruno} and Theorem \ref{thm:oscop}.

%\chema{Changed $B_Q$ to $RH_q$ as already appeared after Definition \ref{def:AinftyHMU} and write so that we can adapt it to elliptic measure} 
As mentioned right after Definition \ref{def:AinftyHMU}, one has that $\omega_L\in A_\infty(\sigma)$ if and only if 
$\omega_L\in RH_q(\sigma)$ for some $q>1$ in the following sense: $\omega_L \ll \sigma$ and the
Radon-Nikodym derivative $\textbf{k}_L:= d\omega_L/d\sigma$ satisfies the reverse H\"older estimate \eqref{eq1.wRH}. We can then define the $RH_q(\sigma)$-characteristic of $\omega_L$ as folows 
\begin{equation}\label{eq1.wRH-char}
[\omega_L]_{RH_q}:=\sup\left(\fint_{\Delta'} \big(\textbf{k}_L^{A(q,r)}\big)^q d\sigma \right)^{\frac1q} \left(\fint_{\Delta'} \textbf{k}_L^{A(q,r)} \,d\sigma\right)^{-1},
\end{equation}
where the sup runs over all $q\in\partial\Omega$, $0<r<\diam (\Omega)$, and all surface balls $\Delta'=B'\cap \pO$ centered at $\pO$ with $B'\subset B(q,r)$. 

The following example, based on the work in \cite{MM} and communicated to us by Bruno Guiseppe  Poggi Cevallos,  illustrates the relationship between the size of the constant in the DKP condition and the $RH_q(\sigma)$-characteristic of elliptic measure. 

%\chema{Changed to $\RR^n_+$ and removed $\Omega$}
\begin{prop}[\cite{MM, P}]\label{p:bruno}
There exist ${\mathcal A}$ and a sequence $\{{\mathcal A}_j\}_j$ of diagonal elliptic matrices with smooth, bounded, real coefficients in $\RR^{n}_+$, uniformly continuous on $\overline{\RR^{n}_+}$, such that ${\mathcal A}_j$ converges to ${\mathcal A}$ uniformly on $\overline{\RR^{n}_+}$ and the following hold: 
\begin{enumerate}[label=\textup{(\arabic*)}, itemsep=0.2cm]
\item 
$\displaystyle{
\sup_{\substack{q\in\RR^{n-1} \\ 0< r<\infty} } \frac{1}{r^{n-1}} \iint_{B(q,r) \cap\RR^{n}_+} 
\bigg(
\sup_{Y\in B(X,\frac{\delta(X)}{2})} |\nabla \mathcal{A}_j(Y)|^2 \delta(Y)\bigg)dX \gtrsim j.
}$

\item For each $q>1$, one has $\omega_j\in RH_q(\sigma)$ with $\displaystyle{\lim_{j\to\infty}[\omega_j]_{RH_q}=\infty}$, where $\omega_j$ denotes the elliptic measure associated with the operator $L_j=-\divg({\mathcal A}_j(\cdot)\nabla)$.

\item The elliptic measure associated with the operator $L=-\divg({\mathcal A}(\cdot)\nabla)$ is singular with respect to the Lebesgue measure on $\partial \RR^{n}_+=\RR^{n-1}$.
\end{enumerate}
\end{prop}
%{\Bl ZZ: In fact, I want something more quantitative for the second item above, involving $j$. Depends on what Bruno gives us.}

On the other hand, we can immediately extend Theorem \ref{thm:main} to a larger and optimal class of elliptic operators, pertaining to the condition on the oscillation of the coefficient matrix:

\begin{corollary}\label{cOsc}
Let $\Omega\subset \RR^n$, $n\ge 3$, be a uniform domain with Ahlfors regular boundary. Let $\wcalA$ be a (not necessarily symmetric) uniformly elliptic matrix on $\Omega$  such that 
	\begin{equation}\label{eq:oscCM}
		\sup_{\substack{q\in\pO\\0< r<\diam(\Omega)} } \frac{1}{r^{n-1}} \iint_{B(q,r) \cap\Omega} \frac{\osc(\wcalA, X)^2}{\delta(X)} \,dX<\infty.
	\end{equation}  where
	$\osc(\wcalA, X) := \sup_{Y,Z\in B(X, \delta(X)/2)}   |\wcalA(Y) -\wcalA(Z)|$. Then the following are equivalent:
\begin{enumerate}[label=\textup{(\alph*)}, itemsep=0.2cm] 
%[label=\textup{(\arabic*)}, itemsep=0.2cm] 	
		\item\label{1-corol-osc} The elliptic measure $\omega_L$ associated with the operator $L=-\divg(\wcalA(\cdot)\nabla)$ is of class $A_\infty$ with respect to the surface measure. 
		\item\label{3-corol-osc} $\pO$ is uniformly rectifiable.
		\item\label{2-corol-osc} $\Omega$ is a chord-arc domain.
		\end{enumerate} 
\end{corollary}

\begin{proof}
Let $\varphi$ be a non-negative radial, smooth bump function supported in the unit ball, such that $\int_{\RR^n} \varphi \,dX = 1$. We define for $X\in\Omega$ and $t\in (0, \delta(X))$
	\[ P_t \, \wcalA(X) := \varphi_t * \wcalA(X) = \iint_{\RR^n} \frac{1}{t^n} \,\varphi\left( \frac{X-Y}{t} \right) \wcalA(Y) \,dY,  \]
	and write for $X\in\Omega$
	\begin{equation}\label{eq:oscdcp}
		\wcalA(X) = P_{\frac{\delta(X)}{4}} \, \wcalA(X) + \left( \Id - P_{\frac{\delta(X)}{4}} \right) \wcalA(X) =: \widetilde{\wcalA}(X) + \left( \wcalA(X) - \widetilde{\wcalA}(X) \right).
	\end{equation} 
It is easy to see that $\widetilde{\wcalA}$ is uniformly elliptic with the same constants as $\wcalA$ and also that for every $X\in\Omega$
	\begin{equation}\label{eq:osc}
		|\nabla \widetilde{\wcalA}(X)| \leq C\, \frac{\osc(\wcalA, X)}{\delta(X)} \qquad\mbox{and}\qquad \sup_{Y\in B(X, \delta(X)/4)} |\wcalA(Y) - \widetilde{\wcalA}(Y)| \leq \osc(\wcalA, X).
	\end{equation}
	Note that under assumption \eqref{eq:oscCM}, the second estimate in \eqref{eq:osc} allows us to invoke \cite[Theorem 1.3]{CHMT} to obtain that $\omega_{L_\wcalA} \in A_\infty(\sigma)$  if and only if $\omega_{L_{\widetilde{\wcalA}}} \in A_\infty(\sigma)$. On the other hand, the first estimate in \eqref{eq:osc} readily implies that $\widetilde{\wcalA}$ satisfies  \ref{H1} and \ref{H2}. Hence,  Theorem \ref{thm:main} applied to $\widetilde{\wcalA}$ gives at once the desired equivalences. 
	We remark that the direction $(c)\implies (a)$ was also proved earlier in \cite[Theorem 2.4]{Rios}.
\end{proof}

Corollary \ref{cOsc} is sharp in terms of the class of operators satisfying \eqref{eq:oscCM}. We recall the following examples that illustrate this fact.

%Recall that elliptic matrices satisfying \eqref{eq:oscCM} is sharp for the $A_\infty$ property, in the sense of the following theorem. Therefore our result is optimal in terms of elliptic operators.
\begin{theorem}{\cite[Theorem 4.11]{FKP}}\label{thm:oscop}
	Suppose $\alpha$ is a non-negative function defined on $\RR^2_+$ satisfying the doubling condition: $\alpha(X) \leq C\alpha(X_0)$ for any $X_0\in \RR^2_+$ and $X\in B(X_0, \delta(X_0)/2)$. Assume that $\alpha^2(x,t) dx \, dt/t$ is not a Carleson measure in the unit square. Then there exists an elliptic operator $L= - \divg(\mathcal{A}(\cdot)\nabla)$ on $\RR^2_+$, such that
	\begin{enumerate}[label=\textup{(\arabic*)}, itemsep=0.2cm]
		\item For any interval $I \subset \RR$ and $T(I)=I\times [0,\ell(I)]$,
		\begin{equation}\label{eq:oscbd}
			\frac{1}{|I|} \iint_{T(I)} \frac{\osc^2(\mathcal{A}(x,t))}{t} dx \, dt \leq C\left[ \frac{1}{|I|} \iint_{T(2I)} \frac{\alpha^2(x,t)}{t} dx \, dt + 1 \right]; 
		\end{equation} 
		\item The elliptic measure $\omega_{L}$ is not of class $A_\infty(dx)$ on the unit interval $[0,1]$.
	\end{enumerate}
\end{theorem}
\begin{remarks}
	The examples above are constructed using quasi-conformal maps in $\RR^2$. In \cite{FKP} the authors show the estimate \eqref{eq:oscbd} holds when $\osc (\mathcal{A,}(x,t))$ is replaced by the oscillation of elliptic matrix $\mathcal{A}(X)$ minus the identity matrix, i.e., $a(X): = \sup_{Y\in B(X, \delta(X)/2 )} |\mathcal{A}(Y) - \Id|$. It is easy to see that for those examples \eqref{eq:oscbd} follows. As in \cite[Theorem 3]{CFK}, one can extend the 2 dimensional examples to $\RR^n_+$ with $n\geq 3$ by using the Laplacian operator in the remaining tangential directions.

%	. (See also previous constructions of singular elliptic measures in \cite{CFK}.) The examples in \cite{FKP} in fact satisfy the estimate \eqref{eq:oscbd} for the oscillation of elliptic matrix $A(X)$ from the identity matrix, i.e. $a(X): = \sup_{Y\in B(X, \delta(X)/2 )} |A(Y) - \Id|$; but \eqref{eq:oscbd} for $\osc(A, (x,t))$ follows easily. Moreover, similar to \cite[Theorem 3]{CFK}, one can easily extend the above examples to $\RR^n_+$ with $n\geq 3$ by using the Laplacian operator in the remaining tangential directions.
\end{remarks}


\begin{thebibliography}{ABHMM-}


\bibitem[AH]{AH} H. Aikawa and K. Hirata, \textit{Doubling conditions for harmonic measure in John domains}.
{Ann. Inst. Fourier (Grenoble)} {\bf 58} (2008), no. 2, 429--445.



%\bibitem[ABHM]{ABHM} M. Akman, M. Badger, S. Hofmann, J.M. Martell, \textit{Rectifiability and elliptic measures on 1-sided NTA domains with Ahlfors-David regular boundaries}. Trans. Amer. Math. Soc. \textbf{369} (2017), no. 8, 2017, 5711--5745.

\bibitem[ABHM]{ABHM} M. Akman, M. Badger, S. Hofmann, and J.M. Martell, \textit{Rectifiability and elliptic measures on 1-sided NTA domains with Ahlfors-David regular boundaries}. Trans. Amer. Math. Soc. \textbf{369} (2017), no. 8, 2017, 5711--5745.


%\bibitem [AGMT]{AGMT} J. Azzam, J. Garnett, M. Mourgoglou, and X. Tolsa, Uniform rectifiability, elliptic measure, square functions and $\epsilon$-approximability, arXiv:1612.02650.
%
%\bibitem[AHMM+]{AHMMMTV} J. Azzam, S. Hofmann, J.M. Martel, S. Mayboroda, M. Mourgoglou, X.Tolsa, and A. Volberg
%\textit{Rectifiability of harmonic measure}
%Geom. Funct. Anal. {\bf 26} (2016), 703-728.

\bibitem[Azz]{Az} J. Azzam, \textit{Semi-uniform domains and a characterization of the $A_\infty$ property for harmonic measure}.
Int. Math. Res. Not. (2019)

\bibitem[AGMT]{AGMT}J. Azzam, J. Garnett, M. Mourgoglou, and X. Tolsa, \textit{Uniform rectifiability, elliptic measure, square functions, and $\epsilon$-approximability via an ACF monotonicity formula}. Preprint, {arXiv:1612.02650}.


%\bibitem[AGMT]{AGMT}J. Azzam, J. Garnett, M. Mourgoglou, and X. Tolsa, \textit{Uniform rectifiability, elliptic measure, square functions, and $\epsilon$-approximability via an ACF monotonicity formula}. Preprint, \textit{arXiv:1612.02650}.

\bibitem[AHM+]{7au}  J.\, Azzam, S.\,Hofmann, J.M. Martell, S.\,Mayboroda, M.\,Mourgoglou, X.\, Tolsa, A.\,Volberg,  \textit{Rectifiability of harmonic measure}. 
Geom. Funct. Anal. 26 (2016), no. 3, 703--728.


\bibitem[AHMMT]{AHMMT} J. Azzam, S. Hofmann, J. M. Martell, M. Mourgoglou, and X. Tolsa, \textit{Harmonic measure and quantitative connectivity: geometric characterization of the $L^p$-solvability of the Dirichlet problem}. Preprint, \textit{arXiv:1907.07102}.


\bibitem[AHMNT]{AHMNT} J. Azzam, S. Hofmann, J. M. Martell, K. Nystr\"om, and T. Toro, \textit{A new characterization of chord-arc domains}. {J. European Math. Soc. (JEMS)}  {\bf 19} (2017), no. 4, 967--981.

\bibitem[AMT]{AMT-char} J. Azzam, M. Mourgoglou, and X. Tolsa, \textit{Harmonic measure and quantitative connectivity: geometric characterization of the $L^p$-solvability of the Dirichlet problem. Part II}. Preprint, \textit{arXiv:1803.07975}.

\bibitem[CFK]{CFK} L.A. Caffarelli, E.B. Fabes, and C.E. Kenig, \textit{Completely singular elliptic-harmonic measures.} Indiana Univ. Math. J. 30 (1981), no. 6, 917--924. 

\bibitem[Car]{Ca} L. Carleson,  Interpolation by bounded analytic functions and the corona problem,
{\it Ann. of Math. (2)} {\bf 76} (1962), 547--559.

\bibitem[CG]{CG} L. Carleson and J. Garnett, Interpolating sequences and separation properties,
{\it J. Analyse Math.} {\bf 28} (1975), 273--299.

\bibitem[CHMT]{CHMT} J. Cavero, S. Hofmann, J.M. Martell, and T. Toro, \textit{Perturbations of elliptic operators in 1-sided chord-arc domains. Part II: Non-symmetric operators and Carleson measure estimates}. Preprint, \textit{arXiv:1908.02268}.

\bibitem[Chr]{Ch} M. Christ,  A $T(b)$ theorem with remarks on analytic
capacity and the Cauchy integral. {\it Colloq. Math.}, \textbf{LX/LXI} (1990), 601--628.

\bibitem[DJK]{DJK} B.E. Dahlberg, D.S. Jerison, and C.E. Kenig, \textit{Area integral estimates for elliptic differential operators with nonsmooth coefficients}. Ark. Mat. \textbf{22} (1984), no. 1, 97--108.


\bibitem[DJ]{DJe} G. David and D. Jerison, \textit{Lipschitz approximation
	to hypersurfaces, harmonic measure, and singular integrals}. {Indiana Univ. Math. J.} {\bf 39} (1990),
no. 3, 831--845.


\bibitem[DS1]{DS1} G. David and S. Semmes, Singular integrals and
rectifiable sets in $\mathbb{R}^n$: Au-dela des graphes lipschitziens.
{\it Asterisque} \textbf{193} (1991).

\bibitem[DS2]{DS2} G. David and S. Semmes, {\it Analysis of and on
Uniformly Rectifiable Sets}. Math. Surveys Monogr. \textbf{38}, AMS
1993.

\bibitem[DKP]{DKP} M. Dindos, C. Kenig, and J. Pipher, \textit{BMO solvability and the $A_{\infty}$ condition for elliptic operators}.  J. Geom. Anal. \textbf{21} (2011), no. 1, 78--95.

\bibitem[FKP]{FKP} R. Fefferman, C. Kenig, and J. Pipher, \textit{The theory of weights and the Dirichlet problem for elliptic equations}. Annals of Math. {\bf 134} (1991), 65-124.

\bibitem[GMT]{GMT} J. Garnett, M. Mourgoglou, and X. Tolsa. \textit{Uniform rectifiability in terms of Carleson measure 
estimates and $\varepsilon$-approximability of bounded harmonic functions}. Duke Math. J. \textbf{167} (2018), no. 8, 1473--1524. 

\bibitem[Hof]{H} Quantitative absolute continuity of harmonic measure and the Dirichlet problem: a survey of recent progress,
{\em Acta Mathematica Sinica, English Series} {\bf 35} (2019) (Special Volume in honor of the 65th birthday of Carlos Kenig),
1011-1026.

\bibitem[HLMN]{HLMN} S. Hofmann, P. Le, J.M. Martell, and K. Nystr\"om, {\it The weak-$A_\infty$ property of harmonic and $p$-harmonic measures}. Anal. PDE \textbf{10} (2017), 513--558. 

\bibitem[HM1]{HM} S. Hofmann and J.M. Martell, {\it Uniform rectifiability and harmonic measure I: Uniform rectifiability implies Poisson kernels in $L^p$}.
Ann. Sci. \'Ecole Norm. Sup. {\bf 47} (2014), no. 3, 577--654.


\bibitem[HM2]{HM4} S. Hofmann, and J.M. Martell, \textit{Uniform Rectifiability and harmonic measure IV: Ahlfors regularity plus Poisson kernels in $L^p$ implies uniform rectifiability}. Preprint,  \textit{arXiv:1505.06499}.

\bibitem[HM3]{HM3} S. Hofmann, and J.M. Martell, \textit{Harmonic measure and quantitative connectivity: geometric characterization of the $L^p$-solvability of the Dirichlet problem. Part I}. Preprint, \textit{arXiv:1712.03696}.

\bibitem[HMM1]{HMM} S. Hofmann, J.M. Martell, and S. Mayboroda,
Uniform rectifiability, Carleson measure estimates,
and approximation of harmonic functions. {\it Duke Math. J.} \textbf{165} (2016), no. 12, 2331--2389.

\bibitem[HMM2]{HMM2} S. Hofmann, J.M. Martell, and S. Mayboroda, \textit{Transference of scale-invariant estimates from Lipschitz to non-tangentially accessible to uniformly rectifiable domains}. Preprint, \textit{arXiv:1904.13116}.

\bibitem[HMMTZ]{HMMTZ} S. Hofmann, J.M. Martell, S. Mayboroda, T. Toro, and Z. Zhao, \textit{Uniform rectifiability and elliptic operators satisfying a Carleson measure 
		condition. Part I: The small constant case}. Preprint, \textit{arXiv:1710.06157}.


\bibitem[HMT1]{HMT1} 	
	S. Hofmann, J.M. Martell, and T. Toro, \textit{$A_\infty$ implies NTA for a class of variable coefficient elliptic operators}. J. Differential Equations 	\textbf{263} (2017), no. 10, 6147--6188.

\bibitem[HMT2]{HMT2} S. Hofmann, J.M. Martell, and T. Toro, \textit{Elliptic operators on non-smooth domains}. Book in preparation.

\bibitem[HMU]{HMU} S. Hofmann, J.M. Martell, and I. Uriarte-Tuero, \textit{Uniform rectifiability and harmonic measure II: Poisson kernels in $L^p$ imply uniform rectifiability}. 
 Duke Math. J. {\bf 163} (2014), no. 8, 1601--1654.
 
 \bibitem[JK]{JK} D. Jerison and C. Kenig,  Boundary behavior of
harmonic functions in nontangentially accessible domains, {\it Adv. in Math.}
\textbf{46} (1982), no. 1, 80--147.

\bibitem[Ken]{Ke}  C. Kenig, {\it Harmonic analysis techniques for second order elliptic boundary value 
problems}. CBMS Regional Conference Series in Mathematics, 83. Published for the Conference 
Board of the Mathematical Sciences, Washington, DC; by the American Mathematical Society, 
Providence, RI, 1994.

\bibitem[KKPT]{KKiPT} C. Kenig, B. Kirchheim, J. Pipher, T. Toro, \textit{Square functions and the $A_{\infty}$ property of elliptic measures}. J. Geom. Anal. \textbf{26} (2016), no. 3, 2383--2383.

\bibitem[KP]{KP} C. Kenig, and J. Pipher, \textit{The Dirichlet problem for elliptic equations with drift terms}. Publ. Mat. {\bf 45} (2001), 199--217.

\bibitem[LM]{LM} J. Lewis and M. Murray,  The method of layer potentials for the heat equation in time-varying domains, {\it Mem. Amer. Math. Soc.} \textbf{114} (1995), no. 545.

\bibitem[LV]{LV} J. L. Lewis and A. Vogel, Symmetry theorems and uniform rectifiability,
{\it Boundary Value Problems} {Vol. 2007} (2007), article ID 030190, 59 pages.

\bibitem[MM]{MM} L. Modica, and S. Mortola, \textit{Construction of a singular elliptic-harmonic measure.} Manuscripta Math. 33 (1980/81), no. 1, 81--98.

\bibitem[NTV]{NTV}  F. Nazarov, X. Tolsa, and A.Volberg, \textit{On the uniform rectifiability of AD-regular measures with bounded Riesz transform operator: the case of codimension 1. }  Acta Math. 213 (2014), no. 2, 237--321.

\bibitem[Pog]{P} B.G. Poggi Cevallos, Failure to slide: a brief note on the interplay between the Kenig-Pipher
  condition and the absolute continuity of elliptic measures, arXiv:1912.10115.

\bibitem[RR]{Rfm} F. and M. Riesz, \"Uber die randwerte einer analtischen funktion,
{\it Compte Rendues du Quatri\`eme Congr\`es des Math\'{e}maticiens Scandinaves}, Stockholm 1916,
Almqvists and Wilksels, Upsala, 1920.

\bibitem[Rio]{Rios} C. Rios, \textit{$L^p$ regularity of the Dirichlet problem for elliptic equations with singular drift.} Publ. Mat. 50 (2006) 475-507.

\bibitem[Sem]{Sem} 
S.~Semmes, \textit{Analysis vs. geometry on a class of rectifiable hypersurfaces in $R^n$}. {Indiana Univ. Math. J.} {\bf 39} (1990), 1005--1035.


\bibitem[Ste]{St} E. M. Stein,
{\em Singular Integrals and Differentiability Properties of Functions}.
Princeton Mathematical Series, No. 30. Princeton University Press, Princeton, N.J., 1970.

\bibitem[Wie]{Wiener} N.\,Wiener,  {\it The Dirichlet problem},  J. Math. Phys.  3  (1924),  pp.
127--146.

\bibitem[Zha]{Zh} Z. Zhao, \textit{BMO solvability and the $A_{\infty}$ condition of the elliptic measure in uniform domains}. J. Geom. Anal. (2018) \textbf{28}, no. 3, 866--908.

\end{thebibliography}
\end{document}